\documentclass[11pt]{article}

\usepackage[letterpaper]{geometry}
\usepackage[dvipsnames]{xcolor}
\usepackage{adjustbox}
\usepackage{setspace}

\usepackage{natbib}
 \bibpunct[, ]{(}{)}{,}{a}{}{,}%
\usepackage{tikz}
\usetikzlibrary{calc}

\usepackage{amsmath, amssymb, bm, mathtools, mathdots, amsthm, tikz-cd,mathtools}
\usepackage{algorithm,algorithmic}
\usepackage[english]{babel}
\usepackage[normalem]{ulem}
\usepackage{amsthm}
\usepackage{array, booktabs, tabularx}
\usepackage{mathrsfs}
\usepackage{bbm}
\usepackage{graphicx, caption, subfig, xcolor}
\usepackage{enumitem}
\usepackage{hyperref}

\graphicspath{{images/}}

\newcommand{\assref}[1]{%
  \hyperref[ass:#1]{\textup{(\ref*{ass:#1})}}%
}

\newcommand{\CVaR}{\mathrm{CVaR}}

\newtheorem{theorem}{Theorem}[section]
\newtheorem{corollary}[theorem]{Corollary}
\newtheorem{lemma}[theorem]{Lemma}
\newtheorem{proposition}[theorem]{Proposition}
\newtheorem{definition}[theorem]{Definition}
\newtheorem{example}[theorem]{Example}  
\newtheorem{remark}[theorem]{Remark}

\newtheorem{assumption}{Assumption}

\hypersetup{hidelinks}

\usepackage{xcolor}

\author{
Hamed Amini
\footnote{Center for Applied Optimization, Department of Industrial and Systems Engineering, University of Florida, Gainesville, FL, USA (\href{mailto:aminil@ufl.edu}{aminil@ufl.edu}).}
\and 
Zhecheng Wu
\footnote{Center for Applied Optimization, Department of Industrial and Systems Engineering, University of Florida, Gainesville, FL, USA (\href{mailto:wuz1@ufl.edu}{wuz1@ufl.edu}).}
}

\begin{document}

\title{Local Weak Limits for Equilibrium and Risk in Economic Networks}

\maketitle

\begin{abstract}
We study equilibrium and risk evaluation in large sparse economic networks
with heterogeneous responses, shocks, and bilateral exposures.
Our approach approximates the distribution of equilibrium outcomes through
local computations on a limiting rooted network. Under marked local weak
convergence, we prove convergence in probability of the empirical
equilibrium distribution in two settings: uniformly contractive responses
and bounded monotone nonexpansive responses whose lower and upper root
iterations coalesce. In the latter setting, the limit is independent of
the measurable finite-network equilibrium selection, even when equilibria
are not unique. Under bounded domain and continuity assumptions, this
convergence extends to law-invariant risk measures, including conditional
value-at-risk, with recursion-depth error bounds under additional
regularity. Numerical experiments on production networks and payment clearing systems
illustrate the accuracy and computational benefits of the local
approximation.

\end{abstract}

\section{Introduction}

Economic networks connect firms, financial institutions, and sectors through
directed and weighted relationships. Their equilibrium outcomes depend not
only on local conditions but also on the states of their counterparties.
In production networks, input--output dependencies transmit disruptions
across firms and sectors. In payment systems, an institution's ability to
meet its obligations depends on the payments it receives from others.
Equity cross-holdings and network games with strategic complementarities
exhibit similar interdependence. These models share a recursive structure:
each node responds to a weighted aggregate of neighboring states and an
exogenous shock.

Although interactions are local, their equilibrium effects can be global.
Shocks propagate along indirect paths and may be amplified through feedback
loops, so evaluating a node's outcome generally requires solving a coupled
system of network-wide fixed-point equations. Repeated equilibrium
computation across network and shock realizations can become costly as the
system grows. Our focus is on approximating the empirical distribution of
these equilibrium outcomes and evaluating risk through the associated
cross-sectional loss distribution.

This raises the central question of the paper:
\emph{under what conditions can equilibrium distributions and their
associated risks in large economic networks be approximated through local
computations?}
Sparsity provides a natural starting point, since each node interacts
directly with relatively few others. However, sparsity alone does not make
equilibrium a local quantity: indirect effects may extend far beyond any
fixed neighborhood. A local approximation therefore requires control of
the influence of distant nodes.

We address this question through marked local weak convergence. Rather
than analyzing the entire network, we study its environment as seen from
a uniformly sampled node, retaining the responses, shocks, and bilateral
exposures relevant to the local recursion. For a fixed number of iterations,
the root outcome depends only on a finite neighborhood. Stability
conditions allow us to pass from local iterations to equilibrium
and establish convergence of the empirical outcome distribution. Under
additional boundedness and continuity assumptions, this convergence extends
to law-invariant risk measures, including conditional value-at-risk (CVaR).
The framework connects large-network equilibrium and risk
evaluation to computations on a limiting rooted network.

\subsection{Related Literature}\label{subsec:literature}

Economic network models frequently characterize equilibrium as a fixed
point of a network interaction map. \citet{Eisenberg2001Systemic}
establish existence of clearing payment vectors and uniqueness under
additional regularity conditions. \citet{ElliottGolubJackson2014}
analyze financial cross-holdings and failure cascades induced by
discontinuous losses in asset values. In production networks,
\citet{Leontief1966InputOutputEconomics} describes intersectoral
dependence through input--output equations, while \citet{Acemoglu2012}
relate the propagation of idiosyncratic shocks to the structure of the
input--output network. \citet{AcemogluOzdaglarTahbazSalehi2016} develop
a common interaction framework encompassing production economies,
financial systems, and network games. We study continuous,
order-preserving specifications of this framework, with heterogeneous
responses, shocks, and primitive exposures encoded as marks.

The probabilistic framework builds on local weak convergence and the
objective method, which relate asymptotic properties of finite networks
to functionals of limiting rooted graphs; see
\citet{BenjaminiSchramm2001,aldous2004objective,Hofstad_2024}.
For directed networks, \citet{garavaglia2020pagerank} introduce a
degree-marked exploration and use local weak convergence to establish
limiting PageRank distributions. \citet{OlveraCravioto2022} develops
fixed-depth couplings between graph explorations and marked
Galton--Watson limits for several locally tree-like models.
Our outgoing exploration retains the complete outgoing exposure
configuration of each explored vertex, making row normalization and
finite-depth interaction updates continuous in the marked local
topology. Local weak convergence has also been used to analyze interacting
processes on sparse networks. \citet{LackerRamananWu2023} establish
continuity of process laws and empirical-measure limits for interacting
Markov chains and diffusions.
\citet{FraimanLinOlveraCravioto2023} couple stochastic recursions on
directed graphs to recursions on marked Galton--Watson trees and give
conditions for convergence to an attracting endogenous solution of a
distributional fixed-point equation. Related local methods are used by
\citet{amini2024dynamics} for cascading losses with discontinuous
updates on locally tree-like networks. Our focus is convergence in
probability of the empirical laws of finite-network equilibria.
The passage from finite-depth local convergence to equilibrium is
justified by uniform truncation estimates in the contractive regime
and by extremal monotone bounds in the nonexpansive regime.

The nonexpansive analysis uses classical order-theoretic fixed-point
methods. Tarski's theorem \citep{Tarski1955} provides least and greatest
fixed-points of order-preserving self-maps on complete lattices;
monotone iteration is a standard constructive approach in ordered
spaces (see \citet{Amann1976}).
In our setting,
continuity of the responses and local finiteness permit passage to
the coordinatewise limits of the extremal iterations. Coalescence at
the limiting root then controls every measurable finite-network
equilibrium selection. On tree limits, this condition is related to
endogeny of recursive tree processes, namely measurability with
respect to the underlying innovations; see
\citet{aldous2005survey}. We use root coalescence as a sufficient
condition for equilibrium-law convergence, rather than as a general
characterization of endogeny.

A preliminary conference version \citet{WuAmini2026LocalApproximation}
establishes finite-type equilibrium and risk limits, including selection
independence under root coalescence and a weighted-offspring matrix
criterion. The present paper extends the framework to Polish type spaces
and allows vertex-dependent response envelopes without a common bound
in the nonexpansive equilibrium theorem. We formulate the coalescence
criterion through a positive operator on \(L^1(\mathcal T,\mu)\), where
\(\mu\) is the root-type law, and verify it for bounded-kernel random graph
limits under an integrable type-dependent envelope. We also derive
operator-based truncation bounds in both stability regimes.

The risk analysis builds on coherent and law-invariant risk
measurement; see \citet{Artzner1999,Kusuoka2001,Frittelli2005}.
\citet{chen2013axiomatic} distinguish aggregation across agents
from risk evaluation across scenarios; our object is a law functional
of the cross-sectional equilibrium-loss distribution of a realized
network. Law invariance alone does not imply continuity under weak
convergence. We therefore impose continuity of the law functional
for qualitative convergence and Wasserstein Lipschitz regularity
for quantitative bounds. On a common compact loss domain, weak convergence is equivalent
to convergence in the \(1\)-Wasserstein distance; see
\citet{villani2009optimal}.
Combined with the variational representation of conditional
value-at-risk in \citet{rockafellar2002cvar_general}, this yields
continuity of CVaR and explicit bounds on its sensitivity
to the loss law.
Our contribution is to
connect these distributional properties to empirical equilibrium
limits and recursion-depth approximation on sparse networks.

\subsection{Primary Contributions}\label{subsec:contributions}

We develop a marked local weak convergence framework for sparse directed
economic networks with heterogeneous types, responses, shocks, and
bilateral exposures. Interaction weights are obtained by normalizing
primitive exposure marks rather than treated as independent network
primitives. By retaining the complete outgoing exposure configuration of
each explored vertex, we establish continuity of normalization and
finite-depth interaction recursions in the marked local topology.

Our main results establish convergence in probability of the empirical
equilibrium-output law for systems of the form
\(
x_i=f_i\!\Bigl(\sum_j W_{ij}x_j+\epsilon_i\Bigr)
\)
under marked local weak convergence and suitable stability conditions.
For uniformly contractive responses, a uniform truncation estimate
controls the passage from finite-depth iterations to the unique
equilibrium. For bounded monotone nonexpansive responses, extremal
iterations yield the same convergence when their limiting root values
coalesce. The conclusion holds for every measurable finite-network
fixed-point selection, without requiring finite-network uniqueness.
On type-indexed branching-process limits, we derive a spectral
coalescence criterion using a weighted offspring operator, which reduces
to a nonnegative matrix for finite type spaces. For bounded-kernel
Poisson limits, we verify this criterion under integrability of the
type-dependent response envelope.

Finally, we connect equilibrium-law convergence to risk evaluation.
Under bounded domain and continuity assumptions, joint type--output
convergence implies convergence of the empirical loss distribution and
associated law-invariant risk measures, including conditional
value-at-risk. Under additional Lipschitz regularity, we derive
recursion-depth error bounds and separate truncation error from
fixed-depth local approximation error. These results justify
approximating equilibrium distributions and risk values through
recursions on limiting rooted networks rather than repeated
full network equilibrium computation.

\smallskip
The rest of the paper is organized as follows.
Section~\ref{sec:model-examples} introduces the network interaction
model and its economic applications.
Section~\ref{sec:marked_loc_top} specifies the marked-network
framework and local convergence assumptions.
Section~\ref{sec:equilibrium-risk} establishes the equilibrium
and risk limit theorems.
Section~\ref{sec:coalescence-approximation} develops coalescence
criteria, verifies the assumptions for directed marked
inhomogeneous random graphs, and derives recursion-depth
approximation bounds.
Section~\ref{sec:numerical} presents numerical experiments on
production networks and payment-clearing systems.
 Section~\ref{sec:conclusion} concludes.
Appendix~\ref{app:technical-details} provides the supporting results on
the marked local topology and continuity of finite-depth recursions,
together with the proof of marked random graph convergence.

\section{Model and Examples}
\label{sec:model-examples}

For each \(n\geq1\), let \(G_n\) be a directed weighted graph on
\([n]:=\{1,\ldots,n\}\). An edge \((i,j)\) indicates that the state
of vertex \(i\) may depend on the state of vertex \(j\). The interaction
weights satisfy \(W_{ij}(n)\geq0\) and
\(\sum_{j=1}^n W_{ij}(n)\leq1\) for every \(i\in[n]\).
Each vertex \(i\) has an exogenous shock
\(\epsilon_i(n)\in\mathbb R\) and a continuous nondecreasing response
\(f_i:\mathbb R\to\mathbb R\); to simplify notation, we suppress the
dependence of \(f_i\) on \(n\).
Write \(W(n)=(W_{ij}(n))_{i,j=1}^n\) and
\(\epsilon(n)=(\epsilon_i(n))_{i=1}^n\).
The index \(n\) identifies the finite network and its size, while the
superscript \(k\) denotes the iteration step. Accordingly, write
\(x^k(n)=(x_i^k(n))_{i=1}^n\). Starting from \(x_i^0(n)=0\), define
\begin{equation}
\label{eq:finite-iteration}
x_i^{k+1}(n)
=
f_i\!\Bigl(
\sum_{j=1}^n W_{ij}(n)x_j^{k}(n)+\epsilon_i(n)
\Bigr),
\qquad i\in[n],\quad k\geq0.
\end{equation}
In particular,
\(x_i^{1}(n)=f_i\bigl(\epsilon_i(n)\bigr)\).
An equilibrium is a vector
\(x^{\star}(n)=(x_i^{\star}(n))_{i=1}^n\) satisfying
\begin{equation}
\label{eq:1}
x_i^{\star}(n)
=
f_i\!\Bigl(
\sum_{j=1}^n W_{ij}(n)x_j^{\star}(n)+\epsilon_i(n)
\Bigr),
\qquad i\in[n].
\end{equation}
This formulation follows the network-interaction framework of
\citet{AcemogluOzdaglarTahbazSalehi2016}. The fixed-point arguments
use only row-substochasticity; for the local-convergence results,
weights are obtained from primitive exposures by the normalization
in Section~\ref{sec:mark_assumption}.

We distinguish two response regimes. In the \emph{contractive regime},
the responses satisfy \(f_i(0)=0\) and
\(\operatorname{Lip}(f_i)\leq L\) for a common
\(L\in[0,1)\), uniformly over vertices and network sizes.
Together with row-substochasticity of \(W(n)\), this makes the
network update a contraction in the supremum norm and ensures a unique
finite-network equilibrium. In the \emph{bounded nonexpansive regime},
each response is nonnegative and bounded, with Lipschitz constant at
most one. The response bound may depend on the vertex and network size.
Monotonicity and boundedness yield least and greatest equilibria,
but not necessarily uniqueness. Our convergence result in this regime
instead requires coalescence of the lower and upper iterations at the
limiting root.

The following examples illustrate the correspondence with standard economic
models and the restrictions required by our analysis. To simplify notation,
we suppress the network index \(n\) and write \(x_i^\star\) for equilibrium
states.

\begin{example}[Financial cross-holdings]
Let \(H_{ij}\) be the fraction of firm \(j\) owned by firm \(i\),
let \(D_{i\ell}\) be firm \(i\)'s holding of primitive asset \(\ell\),
and let \(p_\ell\) be its price. The pre-failure valuation system of
\citet{ElliottGolubJackson2014} is
\begin{equation}
\label{eq:cross-holdings}
x_i^\star
=
\sum_\ell D_{i\ell}p_\ell
+
\sum_{j=1}^n H_{ij}x_j^\star,
\end{equation}
where \(x_i^\star\) denotes the value of firm \(i\).
Set \(H_i:=\sum_j H_{ij}\). For \(H_i>0\), this is
\eqref{eq:1} with \(W_{ij}=H_{ij}/H_i\),
\(f_i(z)=H_i z\), and
\(\epsilon_i=H_i^{-1}\sum_\ell D_{i\ell}p_\ell\).
Application of the contraction theorem requires
\(H_i\leq L\) for a common \(L\in(0,1)\) throughout the graph
sequence, together with its assumptions on the rescaled shocks and
marked local convergence. If \(H_i=0\), take \(W_{ij}=0\),
\(f_i(z)=Lz\), and
\(\epsilon_i=L^{-1}\sum_\ell D_{i\ell}p_\ell\), preserving
\eqref{eq:cross-holdings}. The uniform bound on the ownership row sums
\(H_i\) is an additional assumption; it does not follow from the usual
ownership column constraints.
\end{example}

\begin{example}[Production networks]
In \citet{AcemogluOzdaglarTahbazSalehi2016}, the equilibrium
log-outputs satisfy
\begin{equation}
\label{eq:production-log-linear}
x_i^\star
=
\alpha\sum_{j=1}^n W_{ij}x_j^\star+\alpha\epsilon_i,
\qquad
\sum_{j=1}^n W_{ij}=1,
\qquad
\alpha\in(0,1),
\end{equation}
where \(\epsilon_i\) is the log-productivity shock. This is a
special case of \eqref{eq:1} with \(f_i(z)=\alpha z\), and hence
has contraction constant \(L=\alpha\). Equilibrium-law convergence
additionally requires the shock and marked local-convergence
assumptions of the contraction theorem.
\end{example}

\begin{example}[Quadratic network games]
Consider the following linear-quadratic network game with strategic complementarities, as in \citet{BallesterCalvoArmengolZenou2006}. Let \(g_{ij}\geq0\),
\(g_{ii}=0\), \(\beta\geq0\), and suppose agent \(i\)'s payoff is
\[
u_i(x_i,x_{-i})
=
(\theta_i+\eta_i)x_i
-\frac12x_i^2
+\beta x_i\sum_{j=1}^n g_{ij}x_j,
\]
where \(\eta_i\) is an idiosyncratic perturbation. At an interior
equilibrium, the first-order conditions give
\begin{equation}
\label{eq:quadratic_game_linear}
x^\star=\beta Gx^\star+\theta+\eta,
\qquad G=(g_{ij}).
\end{equation}
Set \(\alpha_i:=\beta\sum_j g_{ij}\). For \(\alpha_i>0\), this is
\eqref{eq:1} with \(W_{ij}=\beta g_{ij}/\alpha_i\),
\(f_i(z)=\alpha_i z\), and
\(\epsilon_i=(\theta_i+\eta_i)/\alpha_i\).
Zero rows are handled as in the cross-holdings example.
A uniform bound \(\alpha_i\leq L<1\) throughout the network sequence
ensures contraction and makes \((I-\beta G)^{-1}\) well defined.
This network multiplier describes the propagation of changes in payoff
intercepts through strategic complementarities. Interiority alone does
not imply the contraction bound, and application of the convergence
theorem also requires its assumptions on the rescaled shocks and marked
local convergence.
\end{example}

\begin{example}[Eisenberg--Noe payment clearing]
Following \citet{Eisenberg2001Systemic}, let \(L_{ij}\geq0\)
be the nominal liability of institution \(i\) to institution \(j\),
with \(L_{ii}=0\), and let \(e_i\geq0\) be its external cash flow.
Augmenting the system with liabilities \(L_{i0}\geq0\) to an outside
creditor, define
\(\bar p_i:=L_{i0}+\sum_j L_{ij}>0\) and
\(\Pi_{ij}:=L_{ij}/\bar p_i\).
Under proportional repayment, the clearing payment vector satisfies
\begin{equation}
\label{eq:balanced-en-clearing}
p_i^\star
=
\min\Bigl\{
\bar p_i,\,
e_i+\sum_{j=1}^n \Pi_{ji}p_j^\star
\Bigr\}.
\end{equation}
To match our creditor-to-debtor orientation, set \(C_{ij}=L_{ji}\)
and write \(x_i^\star:=p_i^\star/\bar p_i\) for the repayment
fraction. Define \(\alpha_i:=\sum_j C_{ij}/\bar p_i\).
For \(\alpha_i>0\), the normalized representation in
\eqref{eq:1} is
\[
W_{ij}=\frac{C_{ij}}{\sum_k C_{ik}},
\qquad
\epsilon_i=\frac{e_i}{\sum_j C_{ij}},
\qquad
f_i(z)=\min\{1,\max\{0,\alpha_i z\}\}.
\]
If \(\alpha_i=0\), take \(W_{ij}=0\), \(\epsilon_i=0\), and
\(f_i(z)=\min\{1,e_i/\bar p_i\}\).
The bounded nonexpansive assumptions hold when
\(\alpha_i\leq1\) for every institution; this is an additional
restriction, not a property of every Eisenberg--Noe network.
A sufficient condition is institution-by-institution interbank balance,
\(\sum_j L_{ji}=\sum_j L_{ij}\), under which
\(\alpha_i=1-L_{i0}/\bar p_i\leq1\).
Without outside liabilities, balance gives \(\alpha_i=1\),
\(W_{ij}=L_{ji}/\bar p_i\), and \(\epsilon_i=e_i/\bar p_i\).
The balanced common-liability networks considered by
\citet{AcemogluOzdaglarTahbazSalehi2016} are a special case.
\end{example}

\section{Marked Networks and Local Convergence}
\label{sec:marked_loc_top}
\label{sec:mark_assumption}

We use marked local weak convergence
(see \citet{garavaglia2020pagerank,LackerRamananWu2023,Hofstad_2024}),
with outgoing exploration consistent with our dependency convention.
We specify the marks and neighborhood restrictions needed for the
equilibrium recursion; supporting metric and continuity arguments
are given in Appendix~\ref{app:technical-details}.

Let \(G=(V,E)\) be a forward-locally-finite directed graph, meaning
that every vertex has finite out-degree; no restriction on in-degrees
is imposed. Fix a vertex \(o\in V\), called the \emph{root}.
Let \(\mathsf M^V\) and \(\mathsf M^E\) be Polish vertex- and
edge-mark spaces, equipped with complete compatible metrics
\(d_{\mathsf M^V}\) and \(d_{\mathsf M^E}\).
A rooted marked directed graph is written as
\((G,o,\mathcal M(G))\), where
\(\mathcal M(G)=((m_v)_{v\in V},(m_e)_{e\in E})
\in(\mathsf M^V)^V\times(\mathsf M^E)^E\).

Two rooted marked directed graphs \((G,o,\mathcal M(G))\) and
\((G',o',\mathcal M(G'))\) are \emph{isomorphic}, denoted by
\((G,o,\mathcal M(G))\cong(G',o',\mathcal M(G'))\), if there
exists a bijection \(\phi:V(G)\to V(G')\) such that:
\(\textnormal{(i)}\) \((i,j)\in E(G)\) if and only if
\((\phi(i),\phi(j))\in E(G')\);
\(\textnormal{(ii)}\) \(\phi(o)=o'\);
\(\textnormal{(iii)}\) \(m_i=m'_{\phi(i)}\) for all \(i\in V(G)\);
and \(\textnormal{(iv)}\)
\(m_{(i,j)}=m'_{(\phi(i),\phi(j))}\) for all \((i,j)\in E(G)\).
Let \(\mathcal G_*\) denote the space of forward-locally-finite
rooted directed graphs with these mark spaces, modulo marked
isomorphism.

\begin{definition}[Outgoing neighborhood]
\label{def:U_neighborhood}
Let \(d_G^+(o,v)\) denote the shortest outgoing directed-path
distance from \(o\) to \(v\), with value infinity if \(v\)
is unreachable. For \(k\geq0\), the depth-\(k\) restriction
\(\mathcal N_{\leq k}^+(G,o)\) retains the original marks on
\[
V_k^+(o):=\{v\in V:d_G^+(o,v)\leq k\},
\qquad
E_k^+(o):=\{(v,u)\in E:d_G^+(o,v)<k\}.
\]
For \(\xi=[G,o,\mathcal M(G)]\in\mathcal G_*\), write
\([\xi]_k\) for the rooted marked isomorphism class of this
restriction, and let \(\mathfrak X_k\) denote the space of
such classes.
\end{definition}

Set \(d_V:=1\wedge d_{\mathsf M^V}\) and
\(d_E:=1\wedge d_{\mathsf M^E}\).
For \(a,b\in\mathfrak X_r\), let
\(\operatorname{Iso}(a,b)\) be the set of root-preserving
isomorphisms between their underlying directed graphs, ignoring
marks. Set \(d_r(a,b)=1\) if this set is empty; otherwise, define
\[
d_r(a,b)
=
\min_{\phi\in\operatorname{Iso}(a,b)}
\max\Bigl\{
\max_v d_V(m_v,m'_{\phi(v)}),
\max_{(v,u)}
d_E(m_{(v,u)},m'_{(\phi(v),\phi(u))})
\Bigr\}.
\]
The maxima range over the vertices and edges of \(a\), respectively;
primed marks belong to \(b\), and an empty edge maximum is zero.
Define
\[
d_{\mathrm{loc}}(\xi,\xi')
=
\sum_{r=0}^{\infty}2^{-r-1}
d_r\bigl([\xi]_r,[\xi']_r\bigr).
\]
This is a pseudometric on \(\mathcal G_*\), since outgoing
exploration observes only the forward component of the root.
Let \(\widetilde{\mathcal G}_*:=\mathcal G_*/\sim\), where
\(\xi\sim\xi'\) if \(d_{\mathrm{loc}}(\xi,\xi')=0\).
We use \(d_{\mathrm{loc}}\) also for the induced metric on the
quotient and \(\xi\) for a generic element of
\(\widetilde{\mathcal G}_*\).
By Proposition~\ref{prop:explored-local-polish},
\((\widetilde{\mathcal G}_*,d_{\mathrm{loc}})\) is Polish.

Local convergence \(\xi_m\to\xi\) means that, for every fixed
\(k\), the restrictions \([\xi_m]_k\) and \([\xi]_k\)
eventually have the same rooted directed shape and admit
isomorphisms under which all corresponding marks converge.
Equivalently, \(d_{\mathrm{loc}}(\xi_m,\xi)\to0\).

For a finite marked network, define the empirical rooted law
\[
\mathcal P_n
:=
\frac1n\sum_{i=1}^n
\delta_{[G_n,i,\mathcal M(G_n)]}.
\]
Write \(\langle\nu,h\rangle:=\int h\,d\nu\).
Thus \(\langle\mathcal P_n,h\rangle\) is an empirical average
on the realized network, whereas \(\mathbb E[\cdot]\) denotes
expectation over graph, mark, and sampling randomness.
Throughout, \(o_n\) is conditionally uniform on \([n]\)
given the marked network.

\begin{definition}[Local weak convergence in probability]
\label{def:lwc}
For a deterministic rooted law \(\mathcal P\), we write
\(\mathcal P_n\overset{\mathbb P}{\Rightarrow}\mathcal P\) if, for every
\(h\in C_b(\widetilde{\mathcal G}_*)\), the space of bounded continuous
functions on \(\widetilde{\mathcal G}_*\),
\[
\langle\mathcal P_n,h\rangle
\xrightarrow{\mathbb P}
\langle\mathcal P,h\rangle
\]
\end{definition}

For deterministic marked networks, this reduces to ordinary
weak convergence. Annealed convergence requires only
\(\mathbb E\langle\mathcal P_n,h\rangle
\to\langle\mathcal P,h\rangle\).
We use convergence in probability because our conclusions
concern empirical distributions on realized networks.

\emph{Specific to our economic model.}
Let \(\mathcal T\) be a Polish type space.
Each vertex \(v\) carries a mark
\(m_v=(\tau_v,\epsilon_v,f_v)\), recording its type, shock, and
response, and each present edge \((v,u)\) carries a positive
primitive exposure \(C_{vu}\).
For a fixed \(L\in[0,1)\), let \(\mathsf F_{\mathrm c}\) consist of
continuous nondecreasing responses with \(f(0)=0\) and
\(\operatorname{Lip}(f)\leq L\), equipped with 
$$
d_{\mathrm c}(f,g)
=
\sum_{q=1}^{\infty}2^{-q}
\Bigl(1\wedge\sup_{|z|\leq q}|f(z)-g(z)|\Bigr).$$  Let \(\mathsf F_{\mathrm{ne}}\) consist of bounded,
nonnegative, nondecreasing \(1\)-Lipschitz responses, equipped
with the uniform metric \(d_{\mathrm{ne}}(f,g)=\|f-g\|_\infty\).
Both spaces are Polish under their respective topologies.
Indeed, \(\mathsf F_{\mathrm c}\) is closed in \(C(\mathbb R)\)
under locally uniform convergence. Every bounded monotone continuous
response extends continuously to \([-\infty,+\infty]\), with its order
topology, and these extensions identify \(\mathsf F_{\mathrm{ne}}\)
with a closed subset of the separable Banach space
\(C([-\infty,+\infty])\) under the uniform norm.
In the nonexpansive regime, the envelope
\(\mathfrak b(f):=\sup_z f(z)\) is continuous because
\(|\mathfrak b(f)-\mathfrak b(g)|\leq\|f-g\|_\infty\);
no common envelope bound is imposed.
The vertex-mark spaces are
\(\mathsf M^V_{\mathrm c}
=\mathcal T\times\mathbb R\times\mathsf F_{\mathrm c}\)
and
\(\mathsf M^V_{\mathrm{ne}}
=\mathcal T\times\mathbb R\times\mathsf F_{\mathrm{ne}}\),
and the edge-mark space is \(\mathsf M^E=(0,\infty)\). The edge-mark space has the compatible complete metric
\(d_{\mathsf M^E}(c,c')=|\log c-\log c'|\).
Together with a complete metric on \(\mathcal T\), these
metrics give complete product metrics on the vertex-mark spaces.

For these mark spaces, the root-type projection
\(\vartheta([G,o,\mathcal M(G)]):=\tau_o\) is continuous.
Consequently, \(\mathcal P_n\overset{\mathbb P}{\Rightarrow}\mathcal P\)
implies \(\mu_n\overset{\mathbb P}{\Rightarrow}\mu\), where
\(\mu_n:=n^{-1}\sum_i\delta_{\tau_i(n)}\) and
\(\mu:=\mathcal P\circ\vartheta^{-1}\).
Type-conditioned laws are treated in
Appendix~\ref{app:type-conditioned-laws}.

For each present edge, set
\(W_{vu}:=C_{vu}/\sum_{w:(v,w)\in E(G)}C_{vw}\);
absent edges have weight zero, and vertices without outgoing
edges have zero network input. The depth-\(k\) restriction
reveals the complete outgoing star of every vertex at distance
less than \(k\). It therefore determines all weights entering
the \(k\)-step root recursion; see
Figure~\ref{fig:horizontal_exploration_normalization}.
Normalization and the finite-depth iterates are continuous
local functionals, as established in
Lemma~\ref{lem:fixed-depth-continuity}.

\begin{figure}[t]
\centering
\includegraphics[width=0.9\columnwidth]
{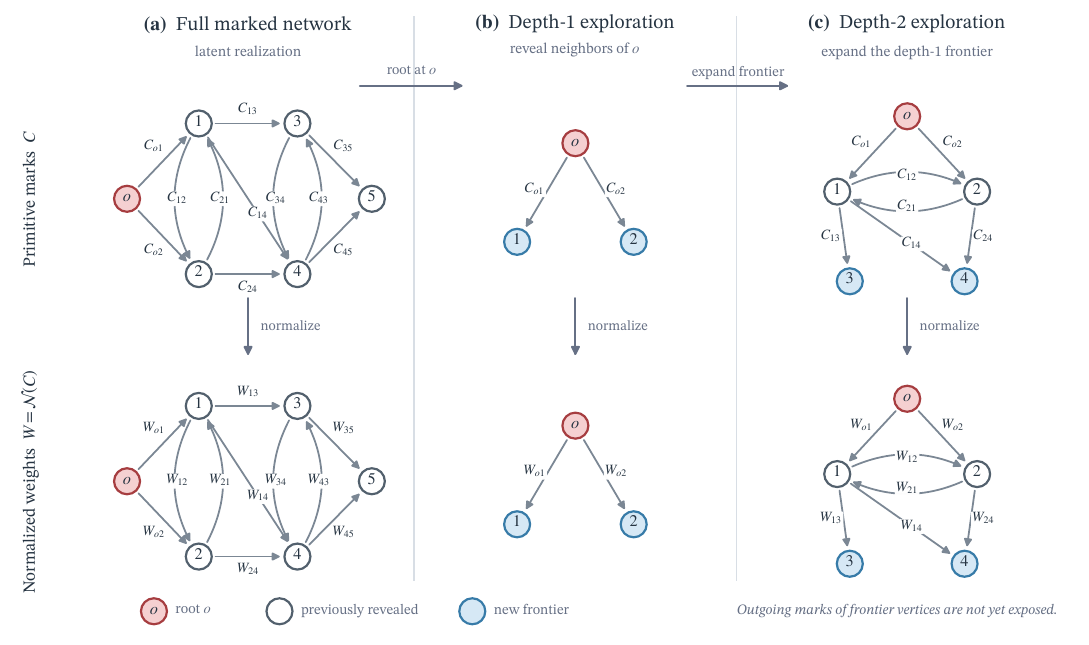}
\caption{Outgoing exploration and normalization. Primitive exposures
(top) and normalized weights (bottom) are shown for the full network
and its depth-\(1\) and depth-\(2\) restrictions.}
\label{fig:horizontal_exploration_normalization}
\end{figure}

The equilibrium results below apply to general locally
convergent marked networks; a directed marked inhomogeneous
random graph construction verifying this assumption is given in
Section~\ref{subsec:directed-marked-irg}.

\section{Equilibrium and Risk Limits}
\label{sec:equilibrium-risk}

We work with the marked networks and normalization of
Section~\ref{sec:marked_loc_top}. The results require marked
local convergence but no particular graph-generation mechanism
or independence assumptions on the graph and marks.
Throughout the limiting network arguments, we take \(G\) to be
the root's forward component, so that
\(V(G)=\bigcup_{r\geq0}V_r^+(o)\).
For a limiting rooted network \((G,o,\mathcal M(G))\), initialize
\(x_v^{0}=0\) and define
\begin{equation}
\label{eq:limit_truncated_recursion}
x_v^{k+1}
=
f_v\!\Bigl(
\sum_{u:(v,u)\in E(G)}W_{vu}x_u^{k}+\epsilon_v
\Bigr),
\qquad v\in V(G),\quad k\geq0.
\end{equation}
When the root iterates converge, write
\(x_o^\star:=\lim_{k\to\infty}x_o^{k}\).
Finite-network iterates and equilibria are denoted by
\(x_i^{k}(n)\) and \(x_i^{\star}(n)\), respectively.
The sampled vertex \(o_n\) is uniform on \([n]\) conditional
on the marked network and any randomness used to select its
equilibrium.

Although each update is local, the equilibrium generally depends
on the entire forward component. The first theorem shows that
marked local convergence determines the limiting equilibrium
distribution under uniform contraction. The contraction and
shock-envelope assumptions control the expected truncation error
uniformly in network size, allowing passage from finite-depth
iterations to equilibrium.

\begin{theorem}[Contractive equilibrium limit]
\label{thm:L<1}
Let \(\{(G_n,\mathcal M(G_n))\}_{n\geq1}\) be finite marked
networks with empirical rooted laws \(\mathcal P_n\). Assume:
\begin{enumerate}[label=(A\arabic*),ref=A\arabic*]
\item\label{ass:A1} \textbf{Uniform contraction.}
For a common \(L\in[0,1)\), every finite-network response
\(f_i\) is continuous and nondecreasing, satisfies
\(f_i(0)=0\), and has Lipschitz constant at most \(L\).
The same conditions hold \(\mathcal P\)-almost surely at every
vertex of the limiting forward component.

\item\label{ass:A2} \textbf{Uniform expected shock envelope.}
The shock vectors satisfy
\[
M_\epsilon
:=
\sup_{n\geq1}
\mathbb E[\|\epsilon(n)\|_\infty]
<\infty.
\]

\item\label{ass:A3} \textbf{Marked local convergence.}
\(\mathcal P_n\overset{\mathbb P}{\Rightarrow}\mathcal P\)
for a deterministic rooted law \(\mathcal P\).
\end{enumerate}
Every finite system has a unique equilibrium
\(x^{\star}(n)\), and the limiting forward component has a
unique bounded equilibrium \(x^\star\), almost surely (a.s.).
For every bounded Lipschitz \(\varphi:\mathbb R\to\mathbb R\),
\[
\frac1n\sum_{i=1}^n
\varphi\bigl(x_i^{\star}(n)\bigr)
\xrightarrow{\mathbb P}
\mathbb E[\varphi(x_o^\star)].
\]
Consequently, \(x_{o_n}^{\star}(n)\Rightarrow x_o^\star\).
\end{theorem}

The strict response-contraction assumption excludes capped payment
maps whose Lipschitz constant is one. We therefore consider bounded
monotone nonexpansive responses, for which lower and upper iterations
bound every equilibrium. The next theorem replaces strict contraction
with coalescence of these iterations at the limiting root, without
requiring finite-network uniqueness.

\begin{theorem}[Monotone nonexpansive equilibrium limit]
\label{thm:Lleq1}
Let \(\{(G_n,\mathcal M(G_n))\}_{n\geq1}\) be finite marked
networks with empirical rooted laws \(\mathcal P_n\). Assume:
\begin{enumerate}[label=(B\arabic*),ref=B\arabic*]
\item\label{ass:B1} \textbf{Marked local convergence.}
\(\mathcal P_n\overset{\mathbb P}{\Rightarrow}\mathcal P\)
for a deterministic rooted law \(\mathcal P\).

\item\label{ass:B2} \textbf{Bounded nonexpansive responses.}
Every finite-network response \(f_i\) is bounded,
nonnegative, nondecreasing, and \(1\)-Lipschitz.
The same conditions hold \(\mathcal P\)-almost surely at every
vertex of the limiting forward component.

\item\label{ass:B3} \textbf{Root coalescence.}
On the limiting network, iterate the update in
\eqref{eq:limit_truncated_recursion} from
\(x_v^{-,0}=0\) and \(x_v^{+,0}=\mathfrak b(f_v)\).
The resulting root iterates satisfy
\[
x_o^{+,k}-x_o^{-,k}\longrightarrow0
\qquad\text{a.s. as }k\to\infty.
\]
\end{enumerate}
Every finite system admits least and greatest equilibria,
and the limiting extremal root iterates have a common limit
\(x_o^\star\).
For every measurable selection \(x^{\star}(n)\) of a
finite-network equilibrium and every bounded Lipschitz
\(\varphi:\mathbb R\to\mathbb R\),
\[
\frac1n\sum_{i=1}^n
\varphi\bigl(x_i^{\star}(n)\bigr)
\xrightarrow{\mathbb P}
\mathbb E[\varphi(x_o^\star)].
\]
In particular, \(x_{o_n}^{\star}(n)\Rightarrow x_o^\star\).
\end{theorem}

\begin{remark}[Equilibrium-selection robustness and scope]
Under Theorem~\ref{thm:Lleq1}, let
\(x^{\star}(n)\) and \(y^{\star}(n)\) be any two measurable
equilibrium selections on the same marked network.
The monotone-sandwich argument also gives
\[
\frac1n\sum_{i=1}^n
\Bigl(1\wedge
|x_i^{\star}(n)-y_i^{\star}(n)|\Bigr)
\xrightarrow{\mathbb P}0.
\]
Indeed, for every fixed \(k\), the left-hand side is bounded by
\[
\frac1n\sum_{i=1}^n
\Bigl(1\wedge
\bigl(x_i^{+,k}(n)-x_i^{-,k}(n)\bigr)\Bigr)
\xrightarrow{\mathbb P}
\mathbb E\Bigl[
1\wedge\bigl(x_o^{+,k}-x_o^{-,k}\bigr)
\Bigr],
\]
by Lemma~\ref{lem:lwc-upper-lower-iterates}.
The limiting expectation vanishes as \(k\to\infty\) by
root coalescence and bounded convergence.
Taking first \(n\to\infty\) and then \(k\to\infty\)
therefore proves the claim.
The truncation by one avoids any need for a common bound
or moment assumption on the response envelopes.
Consequently, for every \(\varepsilon>0\), the fraction of
vertices at which the selections differ by more than
\(\varepsilon\) vanishes in probability.
Neither finite-network uniqueness nor convergence of the
maximum coordinatewise difference is implied.

Neither theorem requires connected finite networks or tree-valued
local limits. In Theorem~\ref{thm:L<1}, uniqueness on the
limiting forward component is asserted within the class of bounded
configurations. In Theorem~\ref{thm:Lleq1}, coalescence is
required only at a \(\mathcal P\)-distributed root; global
uniqueness of the limiting configuration is not assumed.
Proposition~\ref{prop:coalescence-conditions} gives a
branching-process criterion for root coalescence, verified for
bounded-kernel limits in
Corollary~\ref{coro:bounded-kernel-irg-coalescence}.
\end{remark}

Both proofs first apply marked local convergence at a fixed
recursion depth and then remove the truncation. Uniform contraction
controls the approximation error in the first regime; root
coalescence closes the monotone bracket in the second.
Figure~\ref{fig:two-proof-diagrams} summarizes these arguments.

\begin{figure}[t]
\centering
\begin{minipage}[t]{0.46\textwidth}
\vspace{0pt}
\centering
\small\textbf{Contractive regime}
\[
\begin{tikzcd}[
ampersand replacement=\&,
column sep=small,row sep=large,
every label/.append style={font=\scriptsize}]
x_{o_n}^{\star}(n)
  \arrow[r,dashed,"\Rightarrow"]
  \& x_o^\star \\
x_{o_n}^{k}(n)
  \arrow[u,"k\to\infty"]
  \arrow[r,"\Rightarrow"']
  \& x_o^{k} \arrow[u,"k\to\infty"']
\end{tikzcd}
\]
\end{minipage}
\hfill
\begin{minipage}[t]{0.50\textwidth}
\vspace{0pt}
\centering
\small\textbf{Monotone nonexpansive regime}
\[
\begin{tikzcd}[
ampersand replacement=\&,
column sep=small,row sep=normal,
every label/.append style={font=\scriptsize}]
x_{o_n}^{+,k}(n)
  \arrow[r,"\Rightarrow"]
  \& x_o^{+,k} \arrow[d,"\downarrow"] \\
x_{o_n}^{\star}(n)
  \arrow[u,phantom,"\leq" description]
  \arrow[d,phantom,"\geq" description]
  \arrow[r,dashed,"\Rightarrow"]
  \& x_o^\star \\
x_{o_n}^{-,k}(n)
  \arrow[r,"\Rightarrow"']
  \& x_o^{-,k} \arrow[u,"\uparrow"']
\end{tikzcd}
\]
\end{minipage}
\caption{Finite-depth approximation of equilibrium.
Solid horizontal arrows denote convergence as \(n\to\infty\)
at fixed \(k\), and vertical arrows denote limits as
\(k\to\infty\). The nonexpansive inequalities hold for every
\(n\) and \(k\). Dashed arrows are the conclusions of the
limit theorems.}
\label{fig:two-proof-diagrams}
\end{figure}
\subsection{Contractive Equilibria and Truncation}

We first establish the finite- and limiting-network bounds
used in Theorem~\ref{thm:L<1}.

\begin{proposition}[Finite-network equilibrium and truncation]
\label{prop:finite-fp-tail}
Under \assref{A1}, every finite realization has a unique
equilibrium \(x^{\star}(n)\), and
\[
\|x^{\star}(n)-x^{k}(n)\|_\infty
\leq
\frac{L^{k+1}}{1-L}\|\epsilon(n)\|_\infty,
\qquad k\geq0.
\]
Under \assref{A2}, it follows that
\[
\sup_{n\geq1}
\mathbb E\bigl[
|x_{o_n}^{\star}(n)-x_{o_n}^{k}(n)|
\bigr]
\leq \frac{M_\epsilon L^{k+1}}{1-L}
\longrightarrow0.
\]
\end{proposition}

\begin{proof}
Define
\((T_nz)_i=f_i\bigl(
\sum_jW_{ij}(n)z_j+\epsilon_i(n)\bigr)\).
The row-sum bound and \assref{A1} give
\(\|T_nz-T_nz'\|_\infty\leq L\|z-z'\|_\infty\).
The Banach fixed-point theorem yields existence, uniqueness,
and convergence of the iterates.
Since \(f_i(0)=0\), the first increment is bounded by
\(L\|\epsilon(n)\|_\infty\). Hence
\[
\|x^{j+1}(n)-x^{j}(n)\|_\infty
\leq L^{j+1}\|\epsilon(n)\|_\infty.
\]
Summing over \(j\geq k\) gives the pathwise bound.
Taking expectations and using \assref{A2} gives the second.
\end{proof}

At a fixed recursion depth, the root output is a continuous local
observable. The next proposition therefore follows directly from
marked local convergence.

\begin{proposition}[Fixed-depth empirical and sampled-root convergence]
\label{prop:finite-iteration-convergence}
Under \assref{A1} and \assref{A3}, set
\(\Phi_k(\xi):=x_o^{k}\).
For every fixed \(k\geq0\) and \(g\in C_b(\mathbb R)\),
\[
\langle\mathcal P_n,g\circ\Phi_k\rangle
\xrightarrow{\mathbb P}
\langle\mathcal P,g\circ\Phi_k\rangle.
\]
Consequently, \(x_{o_n}^{k}(n)\Rightarrow x_o^{k}\).
\end{proposition}

\begin{proof}
Lemma~\ref{lem:fixed-depth-continuity} gives
\(g\circ\Phi_k\in C_b(\widetilde{\mathcal G}_*)\), so
\assref{A3} proves the first assertion.
The empirical integrals are uniformly bounded, hence also
converge in \(L^1\). Conditional uniform-root sampling gives
\[
\mathbb E[g(x_{o_n}^{k}(n))]
=
\mathbb E\langle\mathcal P_n,g\circ\Phi_k\rangle
\longrightarrow
\mathbb E[g(x_o^{k})].
\]
\end{proof}

To apply the same contraction argument on the limiting forward
component, we first transfer the shock-envelope bound to the limit.

\begin{lemma}[Limiting shock envelope]
\label{lem:limiting-shock-envelope}
Under \assref{A2} and \assref{A3}, the limiting forward
component satisfies
\[
\mathbb E[\|\epsilon\|_\infty]\leq M_\epsilon,
\qquad
\|\epsilon\|_\infty
:=\sup_{r\geq0}\max_{v\in V_r^+(o)}|\epsilon_v|.
\]
In particular, a.s.,  \(\|\epsilon\|_\infty<\infty\).
\end{lemma}

\begin{proof}
For \(r\geq0\) and \(K>0\), the functional
\(H_{r,K}(\xi):=K\wedge\max_{v\in V_r^+(o)}|\epsilon_v|\)
is bounded and continuous. Moreover,
\(\langle\mathcal P_n,H_{r,K}\rangle
\leq\|\epsilon(n)\|_\infty\).
Local convergence and boundedness give
\[
\langle\mathcal P,H_{r,K}\rangle
=
\lim_{n\to\infty}
\mathbb E\langle\mathcal P_n,H_{r,K}\rangle
\leq M_\epsilon.
\]
Letting first \(K\to\infty\) and then \(r\to\infty\),
monotone convergence proves the claim.
\end{proof}

The limiting shock field is thus bounded a.s., so the
update acts on the Banach space of bounded vertex configurations.

\begin{proposition}[Limiting equilibrium and truncation]
\label{prop:infinite-iteration}
Under \assref{A1}--\assref{A3}, the limiting update has,
a.s., a unique fixed-point
\(x^\star\in\ell^\infty(V(G))\). Its zero-initialized
iterates satisfy
\[
\|x^\star-x^{k}\|_\infty
\leq\frac{L^{k+1}}{1-L}\|\epsilon\|_\infty,
\qquad
\mathbb E|x_o^\star-x_o^{k}|
\leq\frac{M_\epsilon L^{k+1}}{1-L}.
\]
Thus \(x_o^{k}\to x_o^\star\) a.s. and in \(L^1\).
\end{proposition}

\begin{proof}
By Lemma~\ref{lem:limiting-shock-envelope},
\(\epsilon\in\ell^\infty(V(G))\) a.s.
Define
\((T_Gz)_v=f_v\bigl(\sum_uW_{vu}z_u+\epsilon_v\bigr)\).
Then
\(\|T_Gz\|_\infty\leq L(\|z\|_\infty+\|\epsilon\|_\infty)\)
and
\(\|T_Gz-T_Gz'\|_\infty\leq L\|z-z'\|_\infty\).
The Banach fixed-point theorem therefore applies on
\(\ell^\infty(V(G))\). The same telescoping estimate as in
Proposition~\ref{prop:finite-fp-tail} gives the pathwise bound,
and the shock-envelope lemma gives the \(L^1\) bound.
The root limit is measurable as a pointwise limit of the
continuous finite-depth functionals.
\end{proof}

We now combine fixed-depth convergence with the two truncation
bounds to prove the contractive limit theorem.

\begin{proof}[Proof of Theorem~\ref{thm:L<1}]
The preceding propositions give the finite and limiting
equilibria. Fix a bounded Lipschitz \(\varphi\) and set
\(a_k:=L^{k+1}/(1-L)\). The truncation estimates imply
\begin{align*}
\mathbb E\Bigl|
\frac1n\sum_{i=1}^n\varphi(x_i^{\star}(n))
-\mathbb E[\varphi(x_o^\star)]
\Bigr|\leq
2\operatorname{Lip}(\varphi)a_kM_\epsilon
+
\mathbb E\Bigl|
\langle\mathcal P_n,\varphi\circ\Phi_k\rangle
-\langle\mathcal P,\varphi\circ\Phi_k\rangle
\Bigr|.
\end{align*}
For fixed \(k\), the final term tends to zero by
Proposition~\ref{prop:finite-iteration-convergence} and
uniform boundedness. Letting \(k\to\infty\) proves
\(L^1\) convergence, and hence convergence in probability,
of the empirical averages. Finally,
\[
\mathbb E[\varphi(x_{o_n}^{\star}(n))]
=
\mathbb E\Bigl[\frac1n\sum_{i=1}^n
\varphi(x_i^{\star}(n))\Bigr]
\longrightarrow\mathbb E[\varphi(x_o^\star)],
\]
which proves sampled-root convergence.
\end{proof}

\subsection{Monotone Bounds and Equilibrium Selection}

Without strict contraction, we compare equilibria with lower and
upper iterations. Their construction uses the coordinatewise order,
and their convergence will provide the bounds needed for
Theorem~\ref{thm:Lleq1}.

Set \(B_i(n):=\mathfrak b(f_i)\) and
\(B_v:=\mathfrak b(f_v)\).
The product intervals
\([0,\mathbf B(n)]:=\prod_{i=1}^n[0,B_i(n)]\)
and
\([0,\mathbf B_G]:=\prod_{v\in V(G)}[0,B_v]\)
are complete lattices under coordinatewise order. We write \(x\preceq y\) if \(x_v\leq y_v\) at every coordinate.
Under \assref{B2}, the corresponding update maps preserve
these intervals and are order preserving. Tarski's theorem \citep{Tarski1955} therefore gives least and
greatest fixed points.
For either a finite graph or the limiting forward component,
write \(T_G\) for the update and
\(\mathcal E_B(G):=\{y\in[0,\mathbf B_G]:T_Gy=y\}\).
The extremal iterations are
\begin{equation}
\label{eq:type-dependent-extremal-iterates}
x_v^{-,0}=0,
\qquad x_v^{+,0}=B_v,
\qquad x^{\pm,k+1}=T_G(x^{\pm,k}).
\end{equation}
On \(G_n\), denote them by \(x^{\pm,k}(n)\).

\begin{proposition}[Extremal monotone iteration]
\label{prop:infinite-monotone-lattice}
Let \(G\) be finite or forward locally finite, and suppose
\(T_G:[0,\mathbf B_G]\to[0,\mathbf B_G]\) is order
preserving with continuous vertex responses.
Then
\[
x^{-,k}\uparrow x^{-,\star},
\qquad x^{+,k}\downarrow x^{+,\star}
\]
coordinatewise, where \(x^{-,\star}\) and \(x^{+,\star}\) are the least and greatest fixed-points. Every equilibrium \(y\) satisfies
\(x^{-,k}\preceq y\preceq x^{+,k}\) for all \(k\). For finite \(G\), convergence also holds in every norm.
Equilibria here are taken in \([0,\mathbf B_G]\), which contains
every fixed-point of the bounded nonnegative response map.
\end{proposition}

\begin{proof}
Since \(T_G0\succeq 0\), \(T_G\mathbf B_G\preceq\mathbf B_G\),
and \(T_G\) is order preserving,
\[
0\preceq x^{-,k}\preceq x^{-,k+1}
\preceq x^{+,k+1}\preceq x^{+,k}\preceq\mathbf B_G.
\]
Both coordinatewise limits therefore exist.
Each update uses finitely many neighbors, so continuity of
\(f_v\) allows passage to the limit and gives
\(T_Gx^{-,\star}=x^{-,\star}\) and
\(T_Gx^{+,\star}=x^{+,\star}\).
Any equilibrium lies between the initial vectors and hence,
by induction, between every pair of iterates. Passing to the
limit proves extremality.
\end{proof}

At each fixed depth, both bounding iterates are continuous local
observables, even though their limits need not depend continuously
on the full marked network.

\begin{lemma}[Local convergence of monotone bounds]
\label{lem:lwc-upper-lower-iterates}
Under \assref{B1} and \assref{B2}, for every fixed
\(k\geq0\) and \(\psi\in C_b(\mathbb R^2)\),
\[
\frac1n\sum_{i=1}^n
\psi\bigl(x_i^{-,k}(n),x_i^{+,k}(n)\bigr)
\xrightarrow{\mathbb P}
\mathbb E\bigl[\psi(x_o^{-,k},x_o^{+,k})\bigr].
\]
Consequently,
\(\bigl(x_{o_n}^{-,k}(n),x_{o_n}^{+,k}(n)\bigr)
\Rightarrow\bigl(x_o^{-,k},x_o^{+,k}\bigr)\).
\end{lemma}

\begin{proof}
By Lemma~\ref{lem:fixed-depth-continuity}, the root maps
\(\Phi_k^\pm(\xi)=x_o^{\pm,k}\) are continuous.
Thus \(\psi(\Phi_k^-,\Phi_k^+)\) is a bounded continuous
rooted observable, and \assref{B1} gives the empirical limit.
Taking expectations and using conditional uniform-root
sampling gives the sampled-root conclusion.
\end{proof}

The remaining step is to let the recursion depth increase.
Root coalescence makes the limiting bracket collapse, while
bounded test functions control the error without a common
envelope bound.

\begin{proof}[Proof of Theorem~\ref{thm:Lleq1}]
Proposition~\ref{prop:infinite-monotone-lattice} gives finite
extremal equilibria; they are measurable limits of iterates.
It also gives the limiting extremal configurations, whose
root values coincide under \assref{B3}.
For any measurable finite-network equilibrium selection,
\(x_i^{-,k}(n)\leq x_i^{\star}(n)
\leq x_i^{+,k}(n)\).
Fix a bounded Lipschitz \(\varphi\) and define the bounded
continuous rooted functional
\[
\Psi_{\varphi,k}(\xi)
:=
\min\Bigl\{
\operatorname{Lip}(\varphi)
\bigl(\Phi_k^+(\xi)-\Phi_k^-(\xi)\bigr),
2\|\varphi\|_\infty
\Bigr\}.
\]
The sandwich and boundedness of \(\varphi\) give
\begin{equation}
\label{eq:finite-iterate-error}
\Bigl|
\frac1n\sum_{i=1}^n\varphi(x_i^{\star}(n))
-\langle\mathcal P_n,\varphi\circ\Phi_k^-\rangle
\Bigr|
\leq\langle\mathcal P_n,\Psi_{\varphi,k}\rangle.
\end{equation}
For fixed \(k\), both rooted integrals on the right and
in the approximation converge in probability to their
\(\mathcal P\)-integrals. Root coalescence and bounded
convergence imply
\[
\langle\mathcal P,\Psi_{\varphi,k}\rangle\longrightarrow0,
\qquad
\langle\mathcal P,\varphi\circ\Phi_k^-\rangle
\longrightarrow\mathbb E[\varphi(x_o^\star)].
\]
First letting \(n\to\infty\) at fixed \(k\), then
\(k\to\infty\), proves empirical convergence in probability.
The truncation by \(2\|\varphi\|_\infty\) avoids any need
for uniform bounds or moment assumptions on the envelopes.
Boundedness of the empirical averages upgrades convergence
to \(L^1\), and conditional uniform-root sampling gives
\(x_{o_n}^{\star}(n)\Rightarrow x_o^\star\).
\end{proof}

\subsection{Joint and Type-Conditional Limits}

The same arguments retain the root type together with its
equilibrium output. This joint limit yields the type-conditional
results below and allows type-dependent losses in the risk analysis.

\begin{corollary}[Joint type--output convergence]
\label{coro:joint-type-output-convergence}
Under either Theorem~\ref{thm:L<1} or
Theorem~\ref{thm:Lleq1}, and for any measurable equilibrium
selection in the nonexpansive case,
\[
\frac1n\sum_{i=1}^n
\delta_{(\tau_i(n),x_i^{\star}(n))}
\overset{\mathbb P}{\Rightarrow}
\operatorname{Law}(\tau_o,x_o^\star)
\qquad\text{on }\mathcal T\times\mathbb R.
\]
\end{corollary}

\begin{proof}
Take a bounded Lipschitz \(g:\mathcal T\times\mathbb R\to\mathbb R\)
under the product metric \(d_{\mathcal T}(s,t)+|z-z'|\).
For fixed \(k\), the maps
\(\xi\mapsto g(\tau_o,\Phi_k(\xi))\) and
\(\xi\mapsto g(\tau_o,\Phi_k^-(\xi))\)
are bounded continuous local observables in their respective
regimes. Since the type is unchanged by truncation,
\[
|g(t,z)-g(t,z')|
\leq\min\{\operatorname{Lip}(g)|z-z'|,2\|g\|_\infty\}.
\]
The truncation or sandwich arguments above therefore yield
\(n^{-1}\sum_i g(\tau_i(n),x_i^{\star}(n))
\xrightarrow{\mathbb P}\mathbb E[g(\tau_o,x_o^\star)]\).
A countable bounded-Lipschitz convergence-determining class
on the Polish space \(\mathcal T\times\mathbb R\), together
with the subsequence criterion for convergence in probability,
gives the measure-valued conclusion.
\end{proof}

Projecting onto the output coordinate gives the empirical output
law and, at continuity points, its tail probabilities.

\begin{corollary}[Empirical output-law and tail convergence]
\label{cor:empirical-tail-convergence}
Under the same hypotheses,
\[
\frac1n\sum_{i=1}^n\delta_{x_i^{\star}(n)}
\overset{\mathbb P}{\Rightarrow}
\operatorname{Law}(x_o^\star).
\]
At every continuity point \(r\) of the limiting output law,
\[
\frac1n\sum_{i=1}^n\mathbf1_{\{x_i^{\star}(n)>r\}}
\xrightarrow{\mathbb P}\mathbb P(x_o^\star>r).
\]
\end{corollary}

\begin{proof}
Project the joint law onto its output coordinate, then apply
the portmanteau theorem to the continuity set \((r,\infty)\),
using the subsequence characterization of convergence in
probability.
\end{proof}

We can also restrict attention to a type set with positive limiting
mass, provided its boundary has zero mass under the limiting type law.

\begin{corollary}[Type-conditional convergence]
\label{coro:multitype_convergence}
Under the same hypotheses, let \(A\subseteq\mathcal T\)
be Borel with \(\mu(A)>0\) and \(\mu(\partial A)=0\).
Conditionally on the marked system and equilibrium selection,
let \(o_n^{(A)}\) be uniform among vertices of type in \(A\)
when \(\mu_n(A)>0\), and choose it arbitrarily otherwise.
Let \((G,o^{(A)},\mathcal M(G))\) have law \(\mathcal P^A\).
Then
\[
x_{o_n^{(A)}}^{\star}(n)\Rightarrow x_{o^{(A)}}^\star.
\]
At every continuity point \(r\) of the conditional output law,
\[
\frac{1}{n\mu_n(A)}
\sum_{i:\tau_i(n)\in A}
\mathbf1_{\{x_i^{\star}(n)>r\}}
\xrightarrow{\mathbb P}
\mathbb P(x_{o^{(A)}}^\star>r),
\]
where the empirical proportion is defined as zero when
\(\mu_n(A)=0\).
\end{corollary}

\begin{proof}
By joint type--output convergence,
\(\mu_n(A)\xrightarrow{\mathbb P}\mu(A)>0\).
Restricting the joint empirical measures to
\(A\times\mathbb R\) and dividing by their masses gives
convergence to the conditional joint law, by the same
continuity-set argument as in
Proposition~\ref{prop:type-set-mixture}.
Projection onto the output coordinate and conditional
uniform sampling prove the first assertion; portmanteau
gives the tail limit. The zero-denominator event has
probability tending to zero.
\end{proof}

\subsection{Risk Functionals of the Sampled Output}
\label{subsec:risk-functionals}
\label{Risk Measure}

We now transform equilibrium outputs into losses and evaluate their
cross-sectional distribution. Let \(\mathcal F_n\) be generated by the
finite network, its marks, and the selected equilibrium vector. Conditional
on \(\mathcal F_n\), sample \(o_n\) uniformly from \([n]\) and define
\[
Y_{o_n}(n)
:=\ell_{\tau_{o_n}(n)}\bigl(x_{o_n}^{\star}(n)\bigr),
\qquad
Y_o:=\ell_{\tau_o}(x_o^\star).
\]
The associated loss laws are
\[
\eta_n
:=\operatorname{Law}\bigl(Y_{o_n}(n)\mid\mathcal F_n\bigr)
=\frac1n\sum_{i=1}^n
\delta_{\ell_{\tau_i(n)}(x_i^{\star}(n))},
\qquad
\eta:=\operatorname{Law}(Y_o).
\]
Thus \(\eta_n\) is random, whereas \(\eta\) is deterministic.
Risk is evaluated from the empirical loss law of each realized network,
without an outer expectation over network realizations.
The notation \(x\) denotes states and \(Y\) denotes losses; when the
type is understood, we abbreviate \(\ell_{\tau_o}(x_o^\star)\)
by \(\ell(x_o^\star)\).

The equilibrium theorems concern weak convergence of output laws.
For risk evaluation, we additionally impose a common bounded loss
domain and continuity of the loss transformation.

\begin{assumption}[Bounded risk domain]
\label{ass:bounded-risk-domain}
In the contractive regime, let \((\bar\epsilon_t)_{t\in\mathcal T}\)
be deterministic nonnegative bounds, measurable in \(t\), with
\(\bar\epsilon:=\sup_t\bar\epsilon_t<\infty\). Almost surely,
\[
|\epsilon_i(n)|\leq\bar\epsilon_{\tau_i(n)}
\quad\text{for all }n,i,
\qquad
|\epsilon_v|\leq\bar\epsilon_{\tau_v}
\quad\text{for all limiting vertices }v.
\]
In the nonexpansive regime, assume a deterministic \(B_*<\infty\)
such that, a.s.,
\(\mathfrak b(f_i)\leq B_*\) and
\(\mathfrak b(f_v)\leq B_*\) for all finite and limiting vertices.
Define the state interval
\[
K_x:=
\begin{cases}
[-\bar\epsilon/(1-L),\,\bar\epsilon/(1-L)],
&\text{in the contractive regime},\\
[0,B_*],&\text{in the nonexpansive regime}.
\end{cases}
\]
The loss map \(\ell:\mathcal T\times\mathbb R\to\mathbb R\),
\(\ell(t,z)=\ell_t(z)\), is jointly continuous and uniformly
bounded on \(\mathcal T\times K_x\).
\end{assumption}

The equilibrium states and the iterates used below lie in \(K_x\).
Choose a compact interval \(K_\ell\) containing
\(\ell(\mathcal T\times K_x)\). All corresponding losses are bounded,
and their laws belong to \(\mathscr P(K_\ell)\), equipped with the
weak topology. Weak convergence in probability of these random laws
has the same meaning as in Definition~\ref{def:lwc}, with test functions
on \(K_\ell\).

\begin{definition}[Law-invariant risk functional]
\label{def:law_inv}
Take \(L^\infty\) on a common atomless probability space, so that
all bounded real-valued loss laws can be realized there.
A functional \(\rho:L^\infty\to\mathbb R\) is law invariant if
\(X\stackrel{d}{=}Y\) implies \(\rho(X)=\rho(Y)\).
We use its law representation
\(\varrho:\mathcal D\to\mathbb R\), for a class
\(\mathcal D\subseteq\mathscr P(K_\ell)\) containing the loss
laws considered, so that
\(\rho(Z)=\varrho(\operatorname{Law}(Z))\).
\end{definition}

Law invariance alone does not imply continuity with respect to the
loss distribution; see e.g.,~\citet{mcneil2015quantitative, shen2026partial}. We impose the following condition separately.

\begin{assumption}[Weak continuity of the law functional]
\label{ass:weak_continuous_risk}
The measures \(\eta_n\) belong to \(\mathcal D\) a.s.,
\(\eta\in\mathcal D\), and \(\varrho\) is continuous at \(\eta\)
under weak convergence on \(\mathcal D\).
The quantities \(\varrho(\eta_n)\) are measurable; for example,
this holds when \(\varrho\) is Borel measurable on \(\mathcal D\).
\end{assumption}

Joint type--output convergence now passes through the loss map.
The additional continuity assumption is needed only for the
subsequent passage from loss laws to risk values.

\begin{theorem}[Loss-law and risk convergence]
\label{thm:risk_convergence_law}
Assume the hypotheses of either Theorem~\ref{thm:L<1} or
Theorem~\ref{thm:Lleq1}, together with
Assumption~\ref{ass:bounded-risk-domain}.
In the nonexpansive case, let \(x^{\star}(n)\) be any measurable
finite-network equilibrium selection. Then
\(\eta_n\overset{\mathbb P}{\Rightarrow}\eta\).
If Assumption~\ref{ass:weak_continuous_risk} also holds, then
\begin{equation}
\label{eq:qualitative-risk-convergence}
\varrho(\eta_n)
\xrightarrow{\mathbb P}
\varrho(\eta)=\rho(Y_o).
\end{equation}
\end{theorem}

\begin{proof}
Set \(\chi_n:=n^{-1}\sum_i
\delta_{(\tau_i(n),x_i^{\star}(n))}\) and
\(\chi:=\operatorname{Law}(\tau_o,x_o^\star)\).
Corollary~\ref{coro:joint-type-output-convergence} gives
\(\chi_n\overset{\mathbb P}{\Rightarrow}\chi\).
Writing \(\ell_\#\nu:=\nu\circ\ell^{-1}\) for the pushforward
of \(\nu\) under \(\ell\), continuity of \(\ell\) gives
\[
\eta_n=\ell_\#\chi_n
\overset{\mathbb P}{\Rightarrow}
\ell_\#\chi=\eta.
\]
Continuity of \(\varrho\) at \(\eta\) then gives
\eqref{eq:qualitative-risk-convergence}.
\end{proof}

The examples used below are mean loss,
\(\rho_{\mathrm{mean}}(Y):=\mathbb E[Y]\), and conditional
value-at-risk \citep{rockafellar2002cvar_general},
\[
\operatorname{CVaR}_\alpha(Y)
:=\inf_{z\in\mathbb R}
\Bigl\{z+\frac{1}{1-\alpha}\mathbb E[(Y-z)_+]\Bigr\},
\qquad \alpha\in(0,1).
\]
We use the same CVaR notation for its induced law functional.
For \(\nu,\nu'\in\mathscr P(K_\ell)\), its variational
representation gives
\[
\bigl|\operatorname{CVaR}_\alpha(\nu)
-\operatorname{CVaR}_\alpha(\nu')\bigr|
\leq\frac{1}{1-\alpha}W_1(\nu,\nu').
\]
Indeed, for any coupling of the two laws, the integrals of
\((y-z)_+\) differ by at most the expected absolute difference
of the coupled losses, uniformly in \(z\).
Taking infima over \(z\) and then over couplings proves the bound.
On a common compact support, weak convergence and \(W_1\)
convergence are equivalent; see \citet{villani2009optimal}.
Thus mean loss and CVaR satisfy the continuity assumption.

\section{Coalescence Criteria and Approximation Bounds}
\label{sec:coalescence-approximation}

We first give a sufficient condition for root coalescence in terms of
weighted influence on the limiting network. We then verify local
convergence and coalescence for directed marked inhomogeneous random
graphs. Finally, we derive recursion-depth bounds for equilibrium risk,
distinguishing truncation from fixed-depth local approximation.

\subsection{Weighted Offspring Operators and Coalescence}
\label{subsec:criterion-coalescence}

The coalescence argument combines a pathwise bound with a
weighted offspring operator that records both
type propagation and normalized exposures; compare the
multitype operator framework in
\citet[Section~3.4]{Hofstad_2024}.
The pathwise bound applies to general forward-locally-finite
graphs, while the operator representation uses the branching
property specified below.

\begin{lemma}[Gap recursion]
\label{lem:gap-subrecursion}
Under the bounded nonexpansive response conditions, define
\(\delta_v^{k}:=x_v^{+,k}-x_v^{-,k}\) and
\(B_v:=\mathfrak b(f_v)\).
Then \(\delta_v^{0}=B_v\) and
\[
0\leq\delta_v^{k+1}
\leq\sum_{u:(v,u)\in E(G)}W_{vu}\delta_u^{k}.
\]
Consequently, \(0\leq\delta_v^{k}\leq H_k^B(v)\), where
\[
H_k^B(v)
:=
\sum_{v=v_0\to v_1\to\cdots\to v_k}
\Bigl(\prod_{j=0}^{k-1}W_{v_jv_{j+1}}\Bigr)B_{v_k},
\qquad H_0^B(v):=B_v.
\]
The sum is over directed walks of length \(k\); an empty
sum is zero.
\end{lemma}

\begin{proof}
Monotonicity orders the two iterates, and the
\(1\)-Lipschitz property of \(f_v\) bounds their difference
by the weighted sum of the neighbor gaps.
Iterating from \(\delta_v^{0}=B_v\) gives the path bound.
\end{proof}

For the operator criterion, assume a measurable type-indexed
family of responses, \(f_v=f_{\tau_v}\), and write
\(B(t):=\mathfrak b(f_t)\).
Let \(\mu\) denote the root-type law. Choose a measurable
version of the weighted offspring kernel
\[
Q_W(s,A)
:=
\mathbb E\Bigl[
\sum_{u:(o,u)\in E(G)}W_{ou}\mathbf1_{\{\tau_u\in A\}}
\,\Big|\,\tau_o=s
\Bigr],
\qquad A\subseteq\mathcal T\text{ Borel}.
\]
Such a version exists because \(\mathcal T\) is Polish and
the random weighted measure has mass at most one.
For measurable \(h:\mathcal T\to[0,\infty]\), define
\((T_Wh)(s):=\int h(t)Q_W(s,dt)\).
Identities involving conditional laws are understood
\(\mu\)-almost everywhere.
The associated mean-envelope quantity is
\begin{equation}
\label{eq:mean-envelope-sequence}
\Gamma_k(B)
:=\int_{\mathcal T}(T_W^kB)(s)\,\mu(ds),
\qquad k\geq0.
\end{equation}
Under the branching assumption below, this equals
\(\mathbb E[H_k^B(o)]\).

For finite types of positive \(\mu\)-mass, \(T_W\) is
represented by the nonnegative matrix \(\mathbf M\), with
\[
M_{st}
=Q_W(s,\{t\})
=
\mathbb E\Bigl[
\sum_{\substack{u:(o,u)\in E(G)\\\tau_u=t}}W_{ou}
\,\Big|\,\tau_o=s
\Bigr].
\]
Thus \(M_{st}\) is the expected total normalized weight
assigned by a type-\(s\) parent to its type-\(t\) children,
not an individual edge weight. With
\(\boldsymbol\mu=(\mu(\{t\}))_t\) and
\(\mathbf B=(B(t))_t\),
\(\Gamma_k(B)=\boldsymbol\mu^\top\mathbf M^k\mathbf B\).
Null types can be omitted when \(T_W\) is well defined
on \(L^1(\mathcal T,\mu)\).

The next proposition turns decay of this mean weighted boundary
mass into almost-sure coalescence. The branching property is used
only to identify the mean through iterates of \(T_W\).

\begin{proposition}[Coalescence from decay of weighted influence]
\label{prop:coalescence-conditions}
Suppose the limiting marked network is a forward-locally-finite
multitype tree with responses
\(f_v=f_{\tau_v}\in\mathsf F_{\mathrm{ne}}\) and the
following branching property: at every vertex, conditional on
its type, its children's types, and its primitive outgoing
exposures, the descendant marked subtrees rooted at the children
are independent and have the same type-conditioned rooted laws
as the original tree.
Assume \(B\in L^1(\mathcal T,\mu)\), and that \(T_W\)
is well defined as a bounded positive operator on
\(L^1(\mathcal T,\mu)\), with spectral radius
\(r_1(T_W)<1\).
Then
\[
x_o^{-,k}\uparrow x_o^\star,
\qquad
x_o^{+,k}\downarrow x_o^\star
\qquad\text{a.s}.
\]
In particular, condition \assref{B3} of
Theorem~\ref{thm:Lleq1} holds.
\end{proposition}

\begin{proof}
The identity
\(H_{k+1}^B(o)=\sum_uW_{ou}H_k^B(u)\), together with
the branching property, yields by induction
\[
\mathbb E[H_k^B(o)\mid\tau_o=s]=(T_W^kB)(s),
\qquad
\mathbb E[H_k^B(o)]
=\Gamma_k(B)=\|T_W^kB\|_{L^1(\mu)}.
\]
For any \(a\in(r_1(T_W),1)\), the spectral-radius formula
gives a finite \(K_a\) such that
\(\|T_W^k\|_{1\to1}\leq K_a a^k\) for all \(k\).
Consequently,
\[
\Gamma_k(B)\leq K_a a^k\|B\|_{L^1(\mu)},
\qquad
\sum_{k=0}^{\infty}\mathbb E[H_k^B(o)]<\infty.
\]
Tonelli's theorem implies \(H_k^B(o)\to0\) a.s.;
the expectation bound also gives convergence in \(L^1\).
Lemma~\ref{lem:gap-subrecursion} now gives root coalescence.
Since every equilibrium lies between the extremal iterates,
all limiting equilibria have the same root coordinate.
\end{proof}

\begin{remark}[Interpreting the spectral criterion]
\label{rem:coalescence-reading}
For finite types, write \(r(\mathbf M)\) for the spectral
radius of the weighted offspring matrix.
If outgoing weights sum to one at nonleaves and zero at leaves,
then \(\sum_t M_{st}
=\mathbb P(\deg^+(o)>0\mid\tau_o=s)\).
A positive leaf probability at every type therefore gives
\(r(\mathbf M)\leq\|\mathbf M\|_\infty<1\).
If the offspring-count mean matrix \(\mathbf K=(K_{st})\), with
\(K_{st}:=\mathbb E[\#\{u:(o,u)\in E,\tau_u=t\}
\mid\tau_o=s]\), has finite entries, then
\(\mathbf M\leq\mathbf K\) entrywise and
\(r(\mathbf M)\leq r(\mathbf K)\).
Subcriticality \(r(\mathbf K)<1\) is therefore sufficient
but not necessary for coalescence: normalized influence may
decay even on a supercritical tree.
At \(r(\mathbf M)=1\), coalescence need not hold.
For example, on an infinite one-child chain with unit weights,
zero shocks, and \(f(z)=\min\{B,\max\{0,z\}\}\), \(B>0\),
the gap remains \(x_o^{+,k}-x_o^{-,k}=B\) for every \(k\).
\end{remark}

\subsection{Directed Marked Inhomogeneous Random Graphs}
\label{subsec:directed-marked-irg}
We now construct directed marked inhomogeneous random graphs (IRGs)
for which the local convergence and coalescence assumptions can be verified.
Let \(\tau(n)=(\tau_i(n))_{i=1}^n\) be a deterministic
type array with empirical law
\(\mu_n:=n^{-1}\sum_i\delta_{\tau_i(n)}\Rightarrow\mu\).
Fix a bounded continuous kernel
\(\kappa:\mathcal T\times\mathcal T\to[0,\infty)\), with
\(\|\kappa\|_\infty\leq\bar\kappa<\infty\).
Conditional on the types, generate independent directed edge
indicators for \(i\ne j\) with
\[
\mathbb P\bigl(A_{ij}(n)=1\mid\tau(n)\bigr)
=
p_{ij}(n)
:=\frac{\kappa(\tau_i(n),\tau_j(n))}{n}\wedge1,
\qquad A_{ii}(n)=0.
\]
Set \(E_n:=\{(i,j):A_{ij}(n)=1\}\) and
\(G_n^\kappa:=([n],E_n)\).

For this construction, responses are type-indexed, with
\(f_s\in\mathsf F_{\mathrm{resp}}\), where
\(\mathsf F_{\mathrm{resp}}\) denotes the response space of
the applicable regime. Let \(\zeta_s\) be the shock law at type
\(s\), and set
\(\mathsf K_s^V(de,df):=\zeta_s(de)\delta_{f_s}(df)\).
Let \(\mathsf K_{s,t}^E\) be the probability law of a positive
primitive exposure on a present edge from type \(s\) to type \(t\).

\begin{assumption}[Feller type-dependent marking kernels]
\label{ass:feller-marking-kernels}
The maps \(s\mapsto\mathsf K_s^V\) and
\((s,t)\mapsto\mathsf K_{s,t}^E\) are weakly continuous.
Conditional on the type array, the vertex marks are independent,
with
\((\epsilon_i(n),f_i)
\sim\mathsf K_{\tau_i(n)}^V\),
and independent of the edge indicators.
Conditional on the types and realized edge set, present-edge
marks are mutually independent, independent of the vertex marks,
and satisfy
\(C_{ij}(n)
\sim\mathsf K_{\tau_i(n),\tau_j(n)}^E\).
\end{assumption}

The assumption includes continuity of \(s\mapsto f_s\) in the
chosen response topology. Let \(\mathcal M_n\) collect the
vertex marks \((\tau_i(n),\epsilon_i(n),f_i)\)
and present-edge exposures \(C_{ij}(n)\).
Weights are obtained by the normalization in
Section~\ref{sec:mark_assumption}.
Thus \(\kappa\) controls link formation and
\(\mathsf K_{s,t}^E\) specifies exposure conditional on a link.

The corresponding local limit replaces each newly explored
neighborhood by independent offspring with the following marked law.

\begin{definition}[Directed marked kernel branching process]
\label{def:directed-marked-kernel-bp}
Set \(\Lambda_s(dt):=\kappa(s,t)\mu(dt)\).
The directed marked kernel branching process
\((\mathsf T_\kappa,o,\mathcal M(\mathsf T_\kappa))\)
has root type \(\tau_o\sim\mu\) and edges directed from
parents to children. For every \(r\geq0\), conditional on
the typed tree through generation \(r\), the offspring-type
point measures of its generation-\(r\) vertices are
independent and satisfy
\[
\Pi_v:=\sum_{u:(v,u)\in E}\delta_{\tau_u}
\sim\operatorname{PRM}(\Lambda_{\tau_v}),
\]
where \(\operatorname{PRM}\) denotes a Poisson random measure.
Since \(\Lambda_s(\mathcal T)\leq\bar\kappa\), every offspring count
is a.s. finite. Atoms are counted with multiplicity and
correspond to distinct children.
Conditional on the entire typed tree, vertex and edge marks
are mutually independent, with respective laws
\(\mathsf K_{\tau_v}^V\) and
\(\mathsf K_{\tau_v,\tau_u}^E\).
\end{definition}

This construction identifies the limiting rooted law. The following
result gives convergence of the empirical rooted measures in probability,
as required by the equilibrium theorems.

\begin{theorem}[Marked local convergence of directed IRGs]
\label{thm:directed-marked-irg-lwc}
Under the preceding construction, suppose that
\(\mu_n\Rightarrow\mu\), \(\kappa\) is bounded and continuous,
and Assumption~\ref{ass:feller-marking-kernels} holds.
Let \(\mathcal P\) be the law of the rooted marked branching
process in Definition~\ref{def:directed-marked-kernel-bp}.
Then, on \(\widetilde{\mathcal G}_*\),
\[
\mathcal P_n
:=\frac1n\sum_{i=1}^n
\delta_{[G_n^\kappa,i,\mathcal M_n]}
\overset{\mathbb P}{\Rightarrow}\mathcal P.
\]
\end{theorem}

The proof is given in Appendix~\ref{sec:lwc-irg-proof}.
The next corollary verifies root coalescence for the nonexpansive
specialization of this limit.

\begin{corollary}[Coalescence for bounded-kernel IRG limits]
\label{coro:bounded-kernel-irg-coalescence}
Let the limiting network be the directed marked kernel
branching process of
Definition~\ref{def:directed-marked-kernel-bp}, with
\(\|\kappa\|_\infty\leq\bar\kappa<\infty\).
If \(f_v=f_{\tau_v}\in\mathsf F_{\mathrm{ne}}\) and
\(B(t)=\mathfrak b(f_t)\in L^1(\mathcal T,\mu)\),
then root coalescence holds.
\end{corollary}

\begin{proof}
The case \(\bar\kappa=0\) is immediate.
Otherwise, set
\(\lambda(s):=\int\kappa(s,t)\mu(dt)\) and
\(c:=1-e^{-\bar\kappa}<1\).
A type-\(s\) parent has
\(D_s\sim\operatorname{Poisson}(\lambda(s))\) children.
Its normalized weights sum to one when \(D_s>0\) and
to zero otherwise, so
\[
(T_W\mathbf1)(s)=1-e^{-\lambda(s)}\leq c.
\]
Since each weight is at most one, for Borel \(A\),
\[
Q_W(s,A)
\leq\mathbb E[\#\{u:(o,u)\in E,\ \tau_u\in A\}
\mid\tau_o=s]
=\int_A\kappa(s,t)\mu(dt)
\leq\bar\kappa\mu(A).
\]
Hence
\(\|T_Wh\|_\infty\leq\bar\kappa\|h\|_{L^1(\mu)}\)
and
\(\|T_Wq\|_\infty\leq c\|q\|_\infty\)
for integrable \(h\) and bounded \(q\).
Because \(\mu\) is a probability measure, \(T_W\) is
bounded on \(L^1(\mu)\), and
\[
\|T_W^k\|_{L^1(\mu)\to L^1(\mu)}
\leq\bar\kappa c^{k-1},\qquad k\geq1.
\]
Thus \(r_1(T_W)\leq c<1\), and
Proposition~\ref{prop:coalescence-conditions} applies.
\end{proof}

A one-point type space with \(\kappa\equiv\lambda\) gives a
directed Erd\H{o}s--R\'enyi graph and a Poisson Galton--Watson
limit of mean \(\lambda\). For finitely many types, this is a
directed stochastic block model. A type-\(s\) parent has independent
Poisson offspring counts, with mean \(\kappa(s,t)\mu(\{t\})\)
for type \(t\).
The general equilibrium theorems also apply to other graph
models satisfying marked local convergence; they do not depend
on the independent-edge construction above.

\subsection{Recursion-Depth Approximation Bounds}
\label{sec:rates}

We separate recursion-depth truncation from fixed-depth local
approximation. Throughout, assume the relevant equilibrium theorem
of Section~\ref{sec:equilibrium-risk} and
Assumption~\ref{ass:bounded-risk-domain}, so all loss laws are
supported on the common compact interval \(K_\ell\).
The bounds do not include Monte Carlo estimation error.
To quantify truncation, we strengthen continuity to the following
uniform Lipschitz conditions.

\begin{assumption}[Quantitative loss and risk regularity]
\label{ass:w1-lipschitz-risk}
There is a finite constant \(\operatorname{Lip}_\ell\) such that
\[
|\ell_t(z)-\ell_t(z')|
\leq\operatorname{Lip}_\ell|z-z'|,
\qquad t\in\mathcal T,\quad z,z'\in K_x.
\]
The risk functional \(\rho\) is monotone and cash invariant
under the loss convention, \(\rho(Z+c)=\rho(Z)+c\).
All loss laws introduced below belong to \(\mathcal D\), and
there is \(L_\varrho<\infty\) such that
\begin{equation}
\label{eq:w1-lipschitz-risk}
|\varrho(\nu)-\varrho(\nu')|
\leq L_\varrho W_1(\nu,\nu'),
\qquad \nu,\nu'\in\mathcal D.
\end{equation}
\end{assumption}

Monotonicity and cash invariance also imply
\(|\rho(Z)-\rho(Z')|\leq\|Z-Z'\|_\infty\) for losses on the
same probability space: with \(c=\|Z-Z'\|_\infty\), use
\(Z\leq Z'+c\) and \(Z'\leq Z+c\).
This bound controls pathwise truncation without the factor
\(L_\varrho\). Mean loss and \(\operatorname{CVaR}_\alpha\)
satisfy the risk assumptions with \(L_\varrho=1\) and
\(L_\varrho=(1-\alpha)^{-1}\), respectively.

\subsubsection{Error decomposition}
\label{subsec:rate-decomposition}

In the contractive regime, define the depth-\(k\) loss laws
\[
\eta_n^k
:=\frac1n\sum_{i=1}^n
\delta_{\ell_{\tau_i(n)}(x_i^{k}(n))},
\qquad
\eta^k:=\operatorname{Law}\bigl(\ell_{\tau_o}(x_o^{k})\bigr).
\]
Thus the superscript \(k\) indexes recursion depth, as for the state
iterates. The triangle inequality gives
\begin{equation}
\label{eq:rate-decomp}
|\varrho(\eta_n)-\varrho(\eta)|
\leq
\underbrace{|\varrho(\eta_n)-\varrho(\eta_n^k)|}_{\beta_{n,k}}
+\underbrace{|\varrho(\eta_n^k)-\varrho(\eta^k)|}_{E_{n,k}}
+\underbrace{|\varrho(\eta^k)-\varrho(\eta)|}_{\beta_k}.
\end{equation}
The truncation biases \(\beta_{n,k}\) and \(\beta_k\) compare
iterates and equilibria on the same network.
The term \(E_{n,k}\) compares finite- and limiting-network
laws at fixed depth. By fixed-depth continuity and marked local
convergence, \(\eta_n^k\overset{\mathbb P}{\Rightarrow}\eta^k\).
On \(K_\ell\), this is equivalent to convergence in \(W_1\), so
\begin{equation}
\label{eq:fixed-depth-rate}
E_{n,k}
\leq L_\varrho W_1\bigl(\eta_n^k,\eta^k\bigr)
\xrightarrow{\mathbb P}0
\qquad\text{for every fixed }k.
\end{equation}
This is a qualitative local-approximation statement; the assumptions
do not specify a finite-\(n\) rate.

\begin{remark}[The finite-network term]
\label{rem:general-irg-finite-n}
A quantitative estimate requires both a local coupling bound and
control of empirical fluctuations. Writing
\(\bar\eta_n^k:=\mathbb E[\eta_n^k]\), one has
\[
W_1\bigl(\eta_n^k,\eta^k\bigr)
\leq
W_1\bigl(\eta_n^k,\bar\eta_n^k\bigr)
+W_1\bigl(\bar\eta_n^k,\eta^k\bigr).
\]
The first term concerns fluctuations around the annealed law;
the second concerns its approximation by the limiting law.
Branching-process couplings are developed in
\citet{BollobasJansonRiordan2007,Hofstad_2024,OlveraCravioto2022}.
For our marked IRG construction, the exploration argument in
Appendix~\ref{sec:lwc-irg-proof} controls collision probabilities.
Rates for type and mark approximation, and for empirical
fluctuations, require additional quantitative assumptions.
We therefore retain \(E_{n,k}=o_{\mathbb P}(1)\) at fixed \(k\).
\end{remark}

\subsubsection{Contractive regime}
\label{subsec:contraction-error}

We first use the uniform contraction estimate, then refine the
limiting bias bound through the weighted branching structure.
Let \(\Delta_k:=\bar\epsilon L^{k+1}/(1-L)\).
The contraction estimates give
\begin{equation}
\label{eq:contraction-output-bound}
|x_o^\star-x_o^{k}|\leq\Delta_k,
\qquad
\max_{i\in[n]}
|x_i^{\star}(n)-x_i^{k}(n)|\leq\Delta_k.
\end{equation}
Comparing the losses on the same limiting network, or at the
same uniformly sampled vertex of a realized finite network,
and using the \(L^\infty\) risk bound yields
$\beta_{n,k}\vee\beta_k \leq\operatorname{Lip}_\ell\Delta_k.$

A second bound incorporates the weighted branching structure.
Assume that the limiting network has the branching property in
Proposition~\ref{prop:coalescence-conditions}, with
\(f_v=f_{\tau_v}\in\mathsf F_{\mathrm c}\).
Set \(L_t:=\operatorname{Lip}(f_t)\) and define
\[
(R_Wh)(s):=L_s(T_Wh)(s),
\qquad
q_\epsilon(t):=L_t\mathbb E[|\epsilon_o|\mid\tau_o=t].
\]
Here \(T_W\) is the weighted offspring kernel action; only
its branching structure, not the nonexpansive response
assumptions, is used in this specialization.
The profile \(q_\epsilon\) is integrable because
\(q_\epsilon(t)\leq L_t\bar\epsilon_t\leq L\bar\epsilon\).
Suppose \(R_W\) is well defined as a bounded positive operator
on \(L^1(\mathcal T,\mu)\), with \(r_1(R_W)<1\), and set
\[
h_\epsilon:=(I-R_W)^{-1}q_\epsilon
=\sum_{j=0}^{\infty}R_W^jq_\epsilon.
\]
For any bounded positive operator \(A\) on \(L^1(\mu)\) and
\(h\geq0\), write
\(\Gamma_k^A(h):=\int_{\mathcal T}(A^kh)(t)\,\mu(dt)\).
In particular, \(\Gamma_k(B)=\Gamma_k^{T_W}(B)\).

To verify the operator bound, set
\(d_j(t):=\mathbb E[|x_o^{j+1}-x_o^{j}|\mid\tau_o=t]\).
Since \(f_t(0)=0\), one has \(d_0(t)\leq q_\epsilon(t)\).
For subsequent increments, the Lipschitz inequality gives
\[
|x_o^{j+2}-x_o^{j+1}|
\leq L_{\tau_o}\sum_u W_{ou}|x_u^{j+1}-x_u^{j}|.
\]
Conditioning on the offspring types and outgoing exposures, then
using the branching property, yields \(d_{j+1}\leq R_Wd_j\).
Induction gives \(d_j\leq R_W^jq_\epsilon\).
Summing the nonnegative increment bounds over \(j\geq k\)
and applying Tonelli's theorem yields
\begin{equation*}
\label{eq:contraction-graph-certificate}
\mathbb E|x_o^\star-x_o^{k}|
\leq\Gamma_k^{R_W}(h_\epsilon)
=\int_{\mathcal T}
\bigl(R_W^k(I-R_W)^{-1}q_\epsilon\bigr)(t)\,\mu(dt).
\end{equation*}
The coupling on the same rooted tree therefore gives
\(W_1(\eta^k,\eta)
\leq\operatorname{Lip}_\ell\Gamma_k^{R_W}(h_\epsilon)\).
Combining the two truncation bounds, we obtain
\begin{equation}
\label{eq:contraction-min-risk-bound}
\beta_k
\leq
\min\Bigl\{
\operatorname{Lip}_\ell\Delta_k,
L_\varrho\operatorname{Lip}_\ell
\Gamma_k^{R_W}(h_\epsilon)
\Bigr\}.
\end{equation}
The first term follows from a pathwise bound and uses monotonicity
and cash invariance. The second uses an expected error and
\(W_1\)-Lipschitz continuity.
For every \(a\in(r_1(R_W),1)\), the spectral-radius formula gives
a finite \(K_a\) such that
\(\Gamma_k^{R_W}(h_\epsilon)
\leq K_a\|h_\epsilon\|_{L^1(\mu)}a^k\).
Thus decay of weighted influence can improve the worst-type
truncation estimate.

Returning to \eqref{eq:rate-decomp},
\begin{equation*}
\label{eq:contraction-finite-limit}
\begin{aligned}
|\varrho(\eta_n)-\varrho(\eta)|
\leq{}&\operatorname{Lip}_\ell\Delta_k+\min\Bigl\{
\operatorname{Lip}_\ell\Delta_k,
L_\varrho\operatorname{Lip}_\ell
\Gamma_k^{R_W}(h_\epsilon)
\Bigr\}+E_{n,k}.
\end{aligned}
\end{equation*}
For fixed \(k\), the last term vanishes in probability as
\(n\to\infty\); the remaining terms vanish as \(k\to\infty\).
This recovers risk convergence by first passing to the local
limit and then removing the truncation.
It also gives, for every \(\varepsilon>0\),
\begin{equation*}
\label{eq:contraction-two-step-limit}
\lim_{k\to\infty}\limsup_{n\to\infty}
\mathbb P\Bigl(
|\varrho(\eta_n^k)-\varrho(\eta)|>\varepsilon
\Bigr)=0.
\end{equation*}

For CVaR, \eqref{eq:contraction-min-risk-bound} becomes
\[
\bigl|\operatorname{CVaR}_\alpha(\eta^k)
-\operatorname{CVaR}_\alpha(\eta)\bigr|
\leq\operatorname{Lip}_\ell
\min\Bigl\{\Delta_k,
\frac{\Gamma_k^{R_W}(h_\epsilon)}{1-\alpha}\Bigr\}.
\]
A depth for which this bound is at most \(\varepsilon\)
certifies the truncation bias of the limiting risk value.
It does not by itself control the finite-network or Monte Carlo error.
The operator term applies to the limiting branching process;
the uniform bound in \eqref{eq:contraction-output-bound}
applies to both finite and limiting networks.

\subsubsection{Nonexpansive regime}
\label{subsec:nonexpansive-error}

In the nonexpansive regime, the width of the monotone bracket
replaces the contraction error. We apply the preceding decomposition
to either bounding recursion.
For \(\sigma\in\{-,+\}\), define
\[
\eta_n^{\sigma,k}
:=\frac1n\sum_{i=1}^n
\delta_{\ell_{\tau_i(n)}(x_i^{\sigma,k}(n))},
\qquad
\eta^{\sigma,k}
:=\operatorname{Law}\bigl(\ell_{\tau_o}(x_o^{\sigma,k})\bigr).
\]
Let \(\beta_{n,k}^\sigma\), \(E_{n,k}^\sigma\), and
\(\beta_k^\sigma\) be the terms in \eqref{eq:rate-decomp}
with \(\eta_n^k,\eta^k\) replaced by
\(\eta_n^{\sigma,k},\eta^{\sigma,k}\).
Write
\(\delta_i^{k}(n):=x_i^{+,k}(n)-x_i^{-,k}(n)\)
and \(\delta_o^{k}:=x_o^{+,k}-x_o^{-,k}\).

The finite-network truncation is controlled by the empirical
bracket width, while the weighted offspring operator bounds its
limiting mean.

\begin{theorem}[Nonexpansive truncation bounds]
\label{thm:rate-nonexpansive}
Assume the hypotheses of Theorem~\ref{thm:Lleq1},
Assumptions~\ref{ass:bounded-risk-domain}
and~\ref{ass:w1-lipschitz-risk}, and the branching,
type-indexing, integrability, and operator hypotheses of
Proposition~\ref{prop:coalescence-conditions}.
In particular, \(B(t)=\mathfrak b(f_t)\in L^1(\mu)\)
and \(r_1(T_W)<1\).
For every measurable finite-network equilibrium selection
and either \(\sigma\in\{-,+\}\),
\[
\beta_{n,k}^\sigma
\leq L_\varrho\operatorname{Lip}_\ell
\frac1n\sum_{i=1}^n\delta_i^{k}(n),
\qquad
\beta_k^\sigma
\leq L_\varrho\operatorname{Lip}_\ell\Gamma_k(B).
\]
Consequently, for every fixed \(k\),
\begin{equation}
\label{eq:nonexpansive-finite-limit}
|\varrho(\eta_n)-\varrho(\eta)|
\leq2L_\varrho\operatorname{Lip}_\ell\Gamma_k(B)
+o_{\mathbb P}(1),
\qquad n\to\infty.
\end{equation}
\end{theorem}

\begin{proof}
The finite and limiting sandwiches give
\(|x_i^{\star}(n)-x_i^{\sigma,k}(n)|
\leq\delta_i^{k}(n)\) and
\(|x_o^\star-x_o^{\sigma,k}|\leq\delta_o^{k}\).
Coupling each iterate and equilibrium on the same network yields
\[
\begin{aligned}
W_1\bigl(\eta_n^{\sigma,k},\eta_n\bigr)
&\leq\operatorname{Lip}_\ell
\frac1n\sum_{i=1}^n\delta_i^{k}(n), \quad
W_1\bigl(\eta^{\sigma,k},\eta\bigr)
\leq\operatorname{Lip}_\ell\mathbb E[\delta_o^{k}]
\leq\operatorname{Lip}_\ell\Gamma_k(B).
\end{aligned}
\]
The last inequality follows from
Lemma~\ref{lem:gap-subrecursion} and the branching
representation of \(\Gamma_k(B)\).
Applying \eqref{eq:w1-lipschitz-risk} proves the two bias bounds.

For fixed \(k\), the functional
\(\Psi_k:=B_*\wedge(\Phi_k^+-\Phi_k^-)\) is bounded
and continuous and agrees with the gap on the supports
of \(\mathcal P_n\) and \(\mathcal P\).
Thus marked local convergence gives
\[
\frac1n\sum_{i=1}^n\delta_i^{k}(n)
=\langle\mathcal P_n,\Psi_k\rangle
\xrightarrow{\mathbb P}\mathbb E[\delta_o^{k}].
\]
Also, \(E_{n,k}^\sigma\xrightarrow{\mathbb P}0\) by
fixed-depth convergence on the common compact loss domain.
The triangle inequality now gives
\eqref{eq:nonexpansive-finite-limit}; its remainder can be
bounded by the nonnegative quantity
\[
L_\varrho\operatorname{Lip}_\ell
\Bigl|\frac1n\sum_i\delta_i^{k}(n)
-\mathbb E[\delta_o^{k}]\Bigr|
+E_{n,k}^\sigma,
\]
which tends to zero in probability.
\end{proof}

The condition \(r_1(T_W)<1\) gives \(\Gamma_k(B)\to0\).
Accordingly, \eqref{eq:nonexpansive-finite-limit} combines
a deterministic limiting truncation bound with a vanishing
fixed-depth remainder. It is not a finite-\(n\) error rate.
No monotonicity of the loss transformations is needed for
these absolute-error bounds.

\section{Numerical Experiments}
\label{sec:numerical}

We examine the local approximation in two network models: a
sector-calibrated production network and a row-normalized
Eisenberg--Noe-type (EN-type) payment system. The first two
subsections study finite-network approximation as the network size
\(n\) increases and compare the runtime and accuracy of
whole-network and branching-process (BP) calculations. The final
subsection examines recursion-depth truncation and the bounds in
Section~\ref{sec:coalescence-approximation}.

In each finite-network replication, we generate an independent graph
and one realization of its vertex and edge marks. We report mean
loss and \(\CVaR_{0.95}\) of the empirical loss law, with
\[
\CVaR_{0.95}(\eta_n)
=
\CVaR_{0.95}\!\Bigl(
\frac1n\sum_{i=1}^n
\delta_{\ell_{\tau_i(n)}(x_i^{\star}(n))}
\Bigr).
\]
This is the conditional empirical risk in
Section~\ref{subsec:risk-functionals}; losses are not pooled
across shock scenarios before computing CVaR.
For production, \(\ell_t(x)=x\), so CVaR measures upper-tail
output rather than adverse low-output risk. The latter would
instead use \(\ell_t(x)=-x\). For the EN-type system,
\(\ell_t(x)=\bar p_t-x\) is payment shortfall.
We use \(x_i^{k}(n)\) and \(x_o^{k}\) for finite and limiting
iterates, and \(\eta^k\) for the limiting depth-\(k\) loss law.
Depth \(k\) means \(k\) applications of the update from the stated
initialization.

For the runtime comparisons, let \(\widehat R_r\) be the CVaR
estimate in replication \(r\) at a given computational budget,
and let \(\widehat R_{\mathrm{large}}\) be the corresponding
large-network reference. We measure accuracy by the root mean squared error (RMSE),
\[
\operatorname{RMSE}
=
\Bigl\{\frac1{N_{\mathrm{rep}}}
\sum_{r=1}^{N_{\mathrm{rep}}}
(\widehat R_r-\widehat R_{\mathrm{large}})^2
\Bigr\}^{1/2}.
\]
The reference is itself a numerical estimate, not an exact
limiting risk value. Timings were recorded on a 10-core arm64
host running macOS~26.5.2, Python~3.13.5, and NumPy~2.2.3,
with numerical-library thread counts fixed at one.
The comparisons concern distributional risk approximation;
BP calculations do not recover the full equilibrium vector of
a prescribed finite network.

The finite generators use Poisson supplier or dependency counts
rather than the Bernoulli construction in
Section~\ref{subsec:directed-marked-irg}.
Self-links are suppressed in both models. EN-type dependencies
are sampled without replacement, whereas production suppliers
are sampled with replacement and repeated draws are merged by
adding their primitive exposures. Since every type has positive
\(\mu\)-mass, the corresponding limiting kernels are
\[
\kappa_{\mathrm{prod}}(s,t)=\frac{8P_{st}}{\mu(\{t\})},
\qquad
\kappa_{\mathrm{EN}}(s,t)=\frac{\lambda_s q_{st}}{\mu(\{t\})},
\]
where \(P_{st}\), \(\lambda_s\), and \(q_{st}\) are specified
below. For fixed exploration size \(K\), repeated-label and
cross-exploration discrepancies have probability \(O(K^2/n)\).
Together with tightness of the explored population, the coupling
argument in Appendix~\ref{sec:lwc-irg-proof} yields the same
multitype branching-process limits.

\subsection{Production-network approximation}
\label{sec:num-production}

\paragraph{Design.}
We use a heterogeneous version of
\eqref{eq:production-log-linear}, with type-dependent response
coefficients and bounded nonnegative productivity increments.
The increments form a reduced-form stress design, not a
calibration of historical log-productivity shocks.
The type law is calibrated from 2023 Census County Business
Patterns (CBP) establishment counts, supplemented by 2023 Bureau of Labor Statistics Quarterly Census of Employment
and Wages (QCEW) counts for
Farms (\texttt{111CA}) and Rail transportation (\texttt{482}).
Input coefficients come from the 2023 Bureau of Economic Analysis summary direct-requirements table
\citet{BEA2023,USCensusCBP2023,BLSQCEW2023}.
The QCEW rail count covers establishments in the state
unemployment-insurance system and is used as a sector-weight
proxy, not a comprehensive rail-establishment count.
Removing imputed-housing and government sectors leaves
\(65\) private-sector types.
For the nonnegative requirement \(a_{ts}\) from supplier
sector \(t\) to buyer sector \(s\), set
\(\alpha_s:=\sum_t a_{ts}\) and \(P_{st}:=a_{ts}/\alpha_s\).
A type-\(s\) establishment draws
\(N_t^{(s)}\sim\operatorname{Poisson}(8P_{st})\) suppliers
of type \(t\), assigning equal primitive exposure to each
supplier draw before merging duplicates and normalizing.

At a type-\(s\) vertex \(v\), draw
\(\widetilde\epsilon_v\sim
\operatorname{Gamma}(r_s^\Gamma,\theta_s^\Gamma)\)
in the shape--scale convention and set
\begin{equation}
\label{eq:num-production-bounded-shock}
b_v:=\min\{\widetilde\epsilon_v,1\},
\qquad
\epsilon_v:=b_v/\alpha_s,
\qquad f_s(z)=\alpha_s z.
\end{equation}
The shapes range from \(1.8\) to \(3.0\), and the scales
from \(0.0314127\) to \(0.0877168\); all type-specific
values are recorded in the replication files.
The parameters are rescaled to give an establishment-weighted
mean first increment of \(0.14\), rather than to estimate
historical sectoral volatility.
The finite-network update is therefore
\begin{equation}
\label{eq:num-production-iteration}
x_i^{k+1}(n)
=
b_i(n)+
\alpha_{\tau_i(n)}
\sum_{j=1}^n W_{ij}(n)x_j^{k}(n),
\qquad x_i^{0}(n)=0.
\end{equation}
This is \eqref{eq:finite-iteration}, since
\(b_i(n)=f_{\tau_i(n)}(\epsilon_i(n))\).
Only the sampled increment is capped; the response remains linear.
The smallest and largest coefficients are approximately
\(0.187558\) and \(0.774077\), respectively. Thus
\(L:=\max_s\alpha_s<1\) and
\(|\epsilon_v|\leq\bar\epsilon:=1/\min_s\alpha_s\).
Displayed calibration constants are rounded; the symbolic bounds
refer to the unrounded coefficients.
Together with marked local convergence, these conditions
verify Theorem~\ref{thm:L<1} and
Assumption~\ref{ass:bounded-risk-domain}.

We use \(50\) logarithmically spaced sizes between \(100\)
and \(10{,}000\), with \(20\) independent graph--mark
realizations per size. Deterministic type counts retain all
\(65\) types, overrepresenting rare types in the smallest
graphs. The discrepancy of these type proportions from
\(\mu\) is \(O(1/n)\).
The fixed BP reference averages five depth-\(32\) calculations
with \(M=2\times10^5\) reported roots. These calculations
reuse type-specific support pools and are finite-particle
approximations, not independent exact samples from the
limiting equilibrium law.

\paragraph{Results.}
Figure~\ref{fig:num-production-finite} shows the finite-network
mean and CVaR approaching the BP reference levels
\(0.249065\) and \(0.562221\), with decreasing
cross-realization dispersion.
As a separate large-network check, \(20\) independent graphs
of size \(n=8{,}432{,}428\) give mean output \(0.249135\)
and \(\CVaR_{0.95}=0.561906\), compared with BP estimates
\(0.249068\) and \(0.561998\) at \(M=10^5\).

\begin{figure}[tbp]
\centering
\includegraphics[width=0.9\linewidth]
{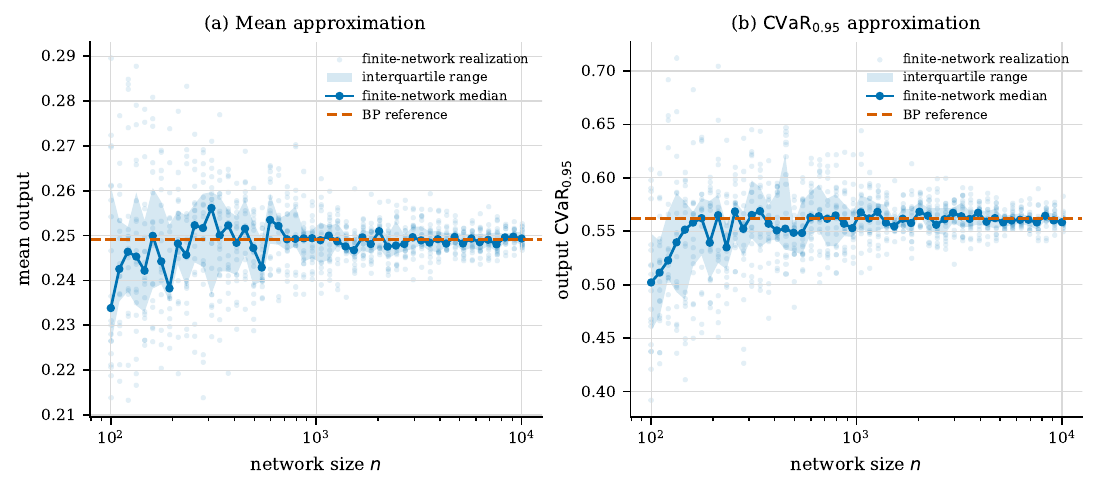}
\caption{Production-network approximation. Points are independent
graph--mark realizations; solid curves and shaded regions show
within-size medians and interquartile ranges. Dashed lines are
numerical BP references. Network-size axes are logarithmic.}
\label{fig:num-production-finite}
\end{figure}

\paragraph{Runtime and accuracy.}
The whole-network design uses one discarded warm-up and eight
measured graphs at each of ten equally spaced sizes through
\(10^6\), followed by \(n=1.1,1.5,2\) million. Size order
is randomized within replicate blocks.
The accuracy reference is the mean CVaR over \(20\)
independent graphs of size \(8{,}432{,}428\).
The whole-network clock includes graph construction, shock
generation, equilibrium solution, and risk evaluation;
the BP clock also includes its final risk evaluation.
All BP timing points use depth \(32\).

Median BP times at
\(M=5{,}000,20{,}000,100{,}000,200{,}000\) are
\(0.207,0.390,1.364,2.572\) seconds, respectively.
On the measured grid, whole-network median runtime first
exceeds these times at
\(n=2\times10^5,3\times10^5,9\times10^5,2\times10^6\).
At \(M=2\times10^5\), BP CVaR RMSE is \(0.001186\).
Figure~\ref{fig:num-production-runtime} relates runtime to
the empirical accuracy attained at each computational budget.

\begin{figure}[tbp]
\centering
\includegraphics[width=0.9\linewidth]
{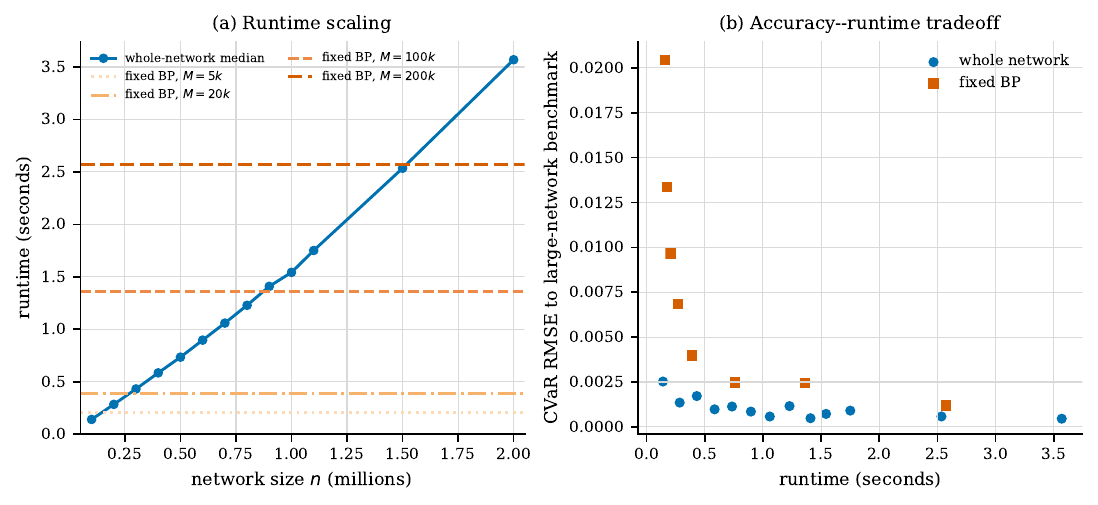}
\caption{Production runtime and accuracy. Panel (a) shows
whole-network median times and fixed-BP times. Panel (b) shows
CVaR RMSE relative to the mean over \(20\) independent
\(8{,}432{,}428\)-node graphs. BP squares correspond to
\(M=1{,}000,2{,}000,5{,}000,10{,}000,20{,}000,50{,}000,
10^5,2\times10^5\); whole-network circles vary \(n\).
Both panels use linear axes.}
\label{fig:num-production-runtime}
\end{figure}

\subsection{Row-normalized EN-type approximation}
\label{sec:num-en}

\paragraph{Design.}
The stylized clearing network has core and peripheral types
with probabilities \((0.1,0.9)\), conditional mean degrees
\((\lambda_1,\lambda_2)=(5.9,1.45)\), and neighbor-type
probabilities
\[
(q_{st})_{s,t=1}^2
=
\begin{pmatrix}
0.0847457627&0.9152542373\\
0.0689655172&0.9310344828
\end{pmatrix}.
\]
At a type-\(s\) vertex, the sampled
\(\operatorname{Poisson}(\lambda_s)\) degree is capped at \(n-1\).
Dependency types are drawn sequentially according to \(q_{s\cdot}\),
renormalized over types with remaining eligible nonself labels.
Labels are sampled uniformly without replacement within each selected
type. The matrix \(q\) is fixed across the experiments.
Nominal payment caps are \((\bar p_1,\bar p_2)=(12,4)\).
Conditional on vertex types, external assets are independent
and satisfy
\[
Z_v\sim\operatorname{Gamma}(r_s,\theta_s),
\qquad a_v=(\gamma_s-Z_v)_+,
\qquad \tau_v=s,
\]
where the Gamma law uses shape--scale parameters
\(\boldsymbol r=(6,2)\),
\(\boldsymbol\theta=(0.5,0.35)\), and
\(\boldsymbol\gamma=(5,1.8)\).
All present dependencies in an outgoing row receive equal primitive
exposures. Thus, if \(D_i(n)\) is the realized out-degree,
\(W_{ij}(n)=1/D_i(n)\) for each selected dependency
when \(D_i(n)>0\); otherwise the network input is zero.
The finite-network and BP calculations use this same weighting rule.
The implemented update is
\begin{equation}
\label{eq:num-en-iteration}
x_i^{k+1}(n)
=
\min\Bigl\{
\bar p_{\tau_i(n)},\,
\Bigl(a_i(n)+\sum_{j=1}^n
W_{ij}(n)x_j^{k}(n)\Bigr)_+
\Bigr\},
\qquad x_i^{0}(n)=0.
\end{equation}
Thus \(\epsilon_v=a_v\) and
\(f_s(z)=\min\{\bar p_s,z_+\}\).
Here \(x_v\) is a payment amount, not a repayment fraction;
the experiment concerns a row-normalized clearing-type map,
not an unrestricted Eisenberg--Noe liabilities network.
The responses satisfy the bounded nonexpansive assumptions,
with \(B_s=\bar p_s\), and the loss is
\(\ell_s(x)=\bar p_s-x\).
For the finite-type Poisson limit, root coalescence follows
from Corollary~\ref{coro:bounded-kernel-irg-coalescence}.


The finite-network experiment uses \(13\) sizes from
\(n=50\) to \(5{,}000\), with \(30\) independent
graph--mark realizations per size. For each graph, the
endowment field is sampled independently conditional on types,
and CVaR is evaluated across its \(n\) node losses.
The dashed BP reference is a prespecified \(60\)-update
population-dynamics calculation with \(2\times10^6\)
particles per type and \(2\times10^6\) reported roots.
Four additional independent-seed calculations assess its
stability. Across the five runs, mean shortfall and CVaR
average \(1.522542\) and \(7.743522\), with standard
deviations \(0.001414\) and \(0.002227\), respectively.
Reported roots reuse type-specific particle pools, so these
are finite-particle reference calculations rather than exact
limiting laws.

\paragraph{Results.}
Figure~\ref{fig:num-en-finite} shows concentration toward the
BP reference values \(1.522734\) for mean shortfall and
\(7.744048\) for CVaR. Between \(n=50\) and \(5{,}000\),
the median absolute gap decreases from \(0.321364\) to
\(0.023319\) for mean shortfall and from \(0.766859\)
to \(0.057367\) for CVaR.
Across the \(390\) finite-network solves, the largest
lower--upper endpoint gap is \(2.33\times10^{-10}\);
the largest lower- and upper-solution fixed-point residuals
are \(7.42\times10^{-11}\) and \(3.96\times10^{-11}\).
These diagnostics concern numerical agreement of the extremal
solutions, not an independent proof of exact uniqueness.

\begin{figure}[ht]
\centering
\includegraphics[width=0.9\linewidth]
{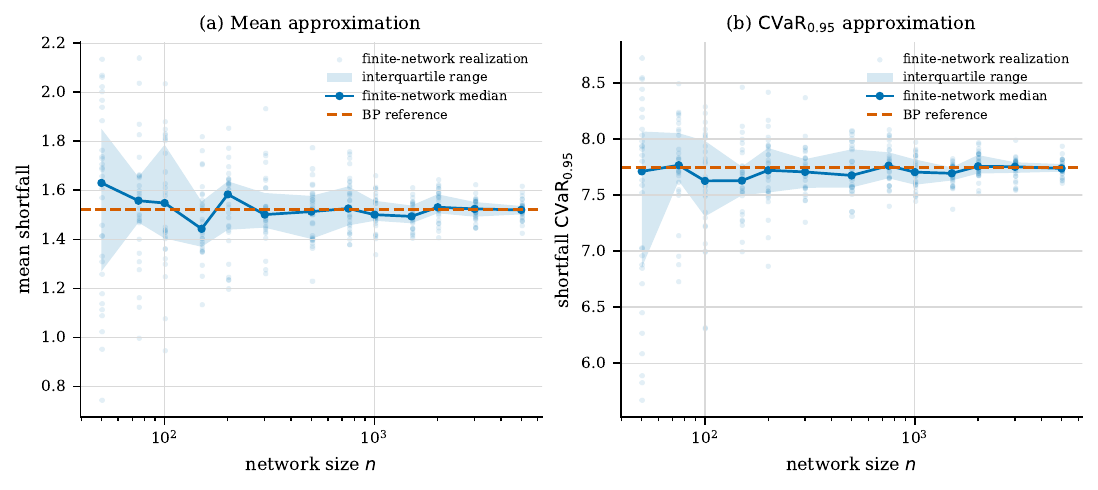}
\caption{EN-type finite-network approximation. Each point is an
independent graph--mark realization. Curves and shaded regions
show within-size medians and interquartile ranges across \(30\)
realizations. Dashed lines are numerical BP references, and
network-size axes are logarithmic.}
\label{fig:num-en-finite}
\end{figure}

\paragraph{Runtime and accuracy.}
The whole-network design uses one discarded warm-up and eight
measured realizations at each of \(16\) sizes: ten equally
spaced sizes from \(10^5\) through \(10^6\), followed by
\(1.25,1.5,1.75,2,2.5,3\) million nodes.
Figure~\ref{fig:num-en-runtime} displays sizes through
\(2\times10^6\); the remaining two are retained as diagnostics.
The accuracy reference is computed from a disjoint set of
\(20\) independent \(3\times10^6\)-node realizations,
giving \(\widehat R_{\mathrm{large}}=7.744923\).
The whole-network clock includes graph generation, the primary
least-payment iteration, and risk evaluation. The two-endpoint
coalescence check is timed separately and excluded.
The BP clock includes depth-\(8\) independent-tree sampling
and risk evaluation, unlike the depth-\(60\) population-dynamics
reference used in the finite-size experiment.

Median BP times at
\(M=2\times10^4,10^5,5\times10^5,10^6,2\times10^6\)
are \(0.185,0.950,4.925,9.867,21.719\) seconds, respectively.
On the measured grid, whole-network median time first exceeds
the first two levels at \(n=10^5\), and the remaining
levels at \(n=2\times10^5,4\times10^5,9\times10^5\).
At \(M=2\times10^6\), BP CVaR RMSE is \(0.003697\).
Accuracy is estimated from \(20\) replications; adjacent
budgets need not produce monotone empirical RMSE values.

\begin{figure}[ht]
\centering
\includegraphics[width=0.9\linewidth]
{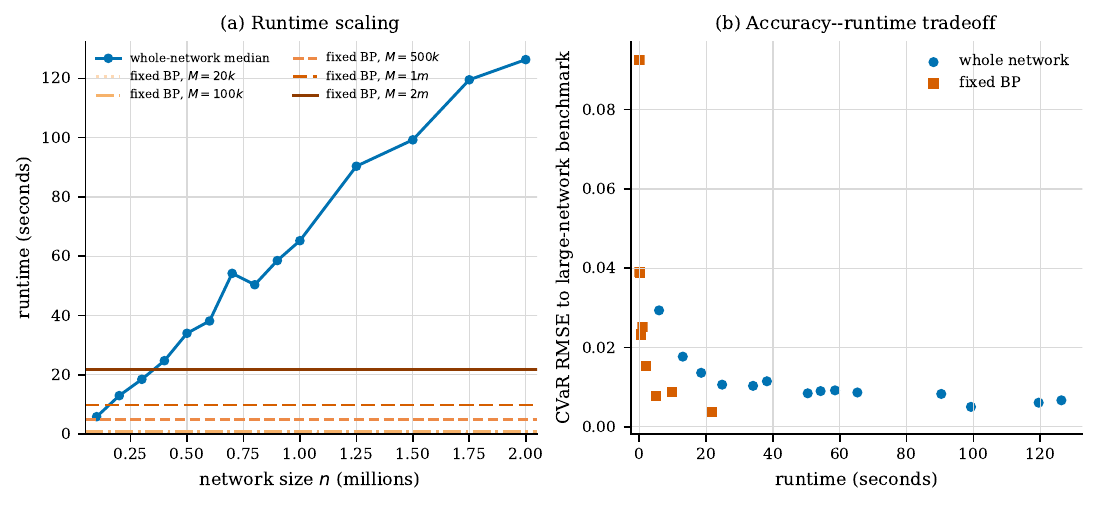}
\caption{EN-type runtime and accuracy. Panel (a) shows
whole-network median times and fixed-BP times. Panel (b) shows
CVaR RMSE relative to the mean over \(20\) independent
\(3\times10^6\)-node realizations. BP squares correspond to
\(M=5{,}000,10{,}000,20{,}000,50{,}000,10^5,2\times10^5,
5\times10^5,10^6,2\times10^6\); whole-network circles vary
\(n\). Both panels use linear axes.}
\label{fig:num-en-runtime}
\end{figure}

\subsection{Recursion-depth truncation}
\label{sec:num-bp-truncation}

We now hold the model fixed and vary the recursion depth.
For production, each truncated output is compared with the
numerical equilibrium on the same finite graph and shock
realization, using \(20\) newly generated
\(10^5\)-node networks. For the EN-type model, coupled lower
and upper recursions are evaluated directly on the BP limit.
The BP certificates concern the exact limiting laws;
finite-particle approximations introduce additional numerical error.

\paragraph{Production network.}
The update \eqref{eq:num-production-iteration} has
\(L_t=\alpha_t\) and \(L\approx0.774077\).
Its first increment satisfies
\(|x_i^{1}(n)-x_i^{0}(n)|\leq d_*:=1\), with the same
bound at the limiting root. The contraction argument therefore gives
\[
\|x^{\star}(n)-x^{k}(n)\|_\infty
\vee |x_o^\star-x_o^{k}|
\leq d_*\frac{L^k}{1-L}.
\]
This improves the generic shock-cap bound
\(\bar\epsilon L^{k+1}/(1-L)\) for this specification.
A second bound retains type-dependent sensitivities and
normalized influence along the branching process.
The total offspring count is Poisson with mean \(8\).
Equal primitive weights give expected total weight
\((1-e^{-8})P_{st}\) from a type-\(s\) parent to type-\(t\)
children, and hence the response-weighted matrix
\[
\mathbf M_L^{\mathrm{prod}}
:=\bigl(\alpha_s(1-e^{-8})P_{st}\bigr)_{s,t},
\qquad r(\mathbf M_L^{\mathrm{prod}})\approx0.410398.
\]
The factor \(1-e^{-8}\) accounts for zero-degree vertices;
it is not the survival probability of the entire tree.
Let \(\boldsymbol\mu=(\mu(\{t\}))_t\) and
\(\mathbf m=(m_t)_t\), where
\(m_t:=\mathbb E[b_o\mid\tau_o=t]=q_\epsilon(t)\).
The equality holds because \(b_o=\alpha_t\epsilon_o\geq0\)
and \(f_t(z)=\alpha_t z\).
With \(\operatorname{Lip}_\ell=1\) and \(L_\varrho=20\),
the finite-type specialization of
\eqref{eq:contraction-min-risk-bound} yields
\begin{equation}
\label{eq:num-production-depth-bound}
\begin{aligned}
\bigl|\CVaR_{0.95}(\eta)-\CVaR_{0.95}(\eta^k)\bigr|\leq
\min\Bigl\{
d_*\frac{L^k}{1-L},\,
20\boldsymbol\mu^\top
(\mathbf M_L^{\mathrm{prod}})^k
(\mathbf I-\mathbf M_L^{\mathrm{prod}})^{-1}\mathbf m
\Bigr\}.
\end{aligned}
\end{equation}
The pathwise bound uses monotonicity and cash invariance of
CVaR, whereas the expected-error bound uses its
\(W_1\)-Lipschitz constant \(20\).
Only the pathwise term applies directly to both finite and
limiting networks. The matrix term bounds the BP truncation
bias; its overlay on finite-network errors is a diagnostic,
not a finite-graph guarantee.
A depth satisfying \eqref{eq:num-production-depth-bound}
certifies the limiting truncation bias, excluding finite-network
and Monte Carlo error.
The smallest certified depths are \(7\) for tolerance
\(0.025\) and \(8\) for tolerance \(0.01\).

The empirical median CVaR bias decreases from \(0.5621\)
at \(k=0\) to \(4.00\times10^{-5}\) at \(k=10\)
and \(5.41\times10^{-9}\) at \(k=20\).
The corresponding certificates are
\(4.426\), \(7.22\times10^{-4}\), and
\(9.80\times10^{-8}\).
At depths \(10\) and \(20\), the matrix term gives the
smaller certificate, reflecting the response-weighted spectral
radius rather than the worst-type factor \(L\).

\paragraph{Payment-clearing system.}
No uniform response Lipschitz constant below one is available.
We instead couple lower and upper recursions using the same
offspring, marks, and child selections, with
\(5\times10^5\) particles per type and eight replications.
The initializations are \(x_v^{-,0}=0\) and
\(x_v^{+,0}=B_{\tau_v}\), where
\(\boldsymbol B^{\mathrm{EN}}=(12,4)^\top\).
Under the equal-exposure specification in
Section~\ref{sec:num-en},
the weighted offspring matrix is
\[
\mathbf M^{\mathrm{EN}}
=\bigl((1-e^{-\lambda_s})q_{st}\bigr)_{s,t=1}^2,
\qquad r(\mathbf M^{\mathrm{EN}})\approx0.7817465.
\]
The matrix-based values below use this equal-exposure
specialization. By Lemma~\ref{lem:gap-subrecursion},
\(\mathbb E[x_o^{+,k}-x_o^{-,k}]
\leq\boldsymbol\mu^\top(\mathbf M^{\mathrm{EN}})^k
\boldsymbol B^{\mathrm{EN}}\).
Since shortfall is decreasing in payments, the limiting risk
lies between the risks of the upper and lower payment iterates.
Their width satisfies
\begin{equation}
\label{eq:num-en-depth-bound}
0\leq
\CVaR_{0.95}(\eta^{-,k})
-\CVaR_{0.95}(\eta^{+,k})
\leq
20\boldsymbol\mu^\top
(\mathbf M^{\mathrm{EN}})^k\boldsymbol B^{\mathrm{EN}}.
\end{equation}
Here \(\eta^{\pm,k}\) are the loss laws generated by
\(x_o^{\pm,k}\), and
\(\boldsymbol\mu=(0.1,0.9)^\top\), so
\(\boldsymbol\mu^\top\boldsymbol B^{\mathrm{EN}}=4.8\).
The certificate is \(20\Gamma_k(B)\), using the
type-dependent envelope rather than a common scalar cap.
The estimated coupled BP risk bracket decreases from \(12\)
at \(k=0\) to \(4.47\times10^{-5}\) at \(k=10\)
and \(1.09\times10^{-12}\) at \(k=20\), compared with
bounds \(96.000\), \(7.849\), and \(0.669\).
The spectral certificate is substantially more conservative
than the observed bracket.

\begin{figure}[tbp]
\centering
\includegraphics[width=0.9\linewidth]
{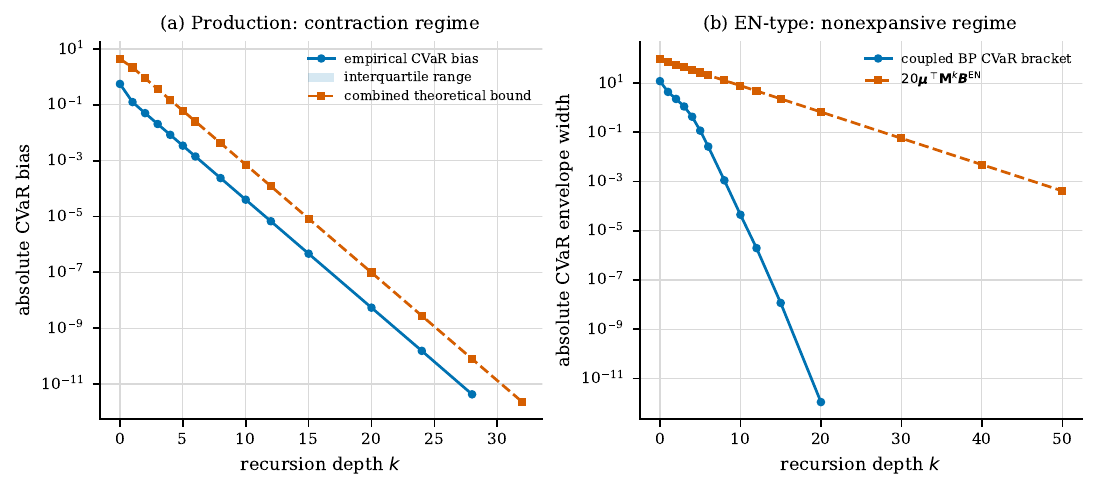}
\caption{Recursion-depth truncation. Panel (a) shows production
median CVaR bias and interquartile ranges over \(20\) new
\(10^5\)-node realizations, with the limiting-BP certificate
\eqref{eq:num-production-depth-bound}; the finite-network
comparison with its matrix term is diagnostic. Panel (b) shows
the estimated coupled EN-type risk bracket and the equal-exposure
certificate \(20\boldsymbol\mu^\top(\mathbf M^{\mathrm{EN}})^k
\boldsymbol B^{\mathrm{EN}}\), with
\(\boldsymbol B^{\mathrm{EN}}=(12,4)^\top\).
Vertical axes are logarithmic.}
\label{fig:num-bp-truncation}
\end{figure}

The experiments support finite-network convergence to the local
BP reference in both stability regimes. The runtime comparisons
quantify the model-dependent accuracy--cost tradeoff, while the
depth experiments distinguish practical truncation behavior from
its analytical bounds. In particular, the EN-type bound guarantees
decay but is conservative at the tested depths.

\section{Conclusion}
\label{sec:conclusion}

We developed a local approximation framework for equilibrium and risk
evaluation in large sparse directed economic networks. Under marked
local weak convergence, we established convergence in probability of
the empirical equilibrium distribution in both the contractive and
monotone nonexpansive regimes. Under additional bounded-domain and continuity assumptions,
these results imply convergence of cross-sectional loss distributions
and law-invariant risk measures, including conditional value-at-risk.

On branching-process limits, weighted offspring operators provide
spectral criteria for root coalescence and quantitative recursion-depth
bounds. We verified coalescence for bounded-kernel directed random-graph
limits under an integrable type-dependent response envelope.
The approximation bounds distinguish truncation error from fixed-depth
local approximation error and show how weighted network structure can
improve worst-case contraction estimates. Experiments on production
networks and a row-normalized payment-clearing system illustrate
convergence toward local-limit benchmarks and the accuracy--computation
tradeoff of local calculations. They also show that sufficient spectral
bounds can be conservative relative to observed truncation errors.

Several extensions are natural. Quantitative finite-network estimates
and sampling-error bounds would complement the recursion-depth analysis
and support the joint selection of neighborhood depth and simulation
budget. Richer financial models could incorporate price-mediated
feedback from fire sales, requiring control of the interaction between
local exposures and aggregate asset prices
\citep{amini2025fire,amini2016uniqueness}.
Another direction is targeted network intervention
\citep{galeotti2020targeting}. Building on our local intervention
framework for large sparse production networks \citep{wu2026sampled},
extending these methods to nonlinear clearing and fire-sale systems
would require equilibrium stability and approximation guarantees
uniform over admissible policies, providing a basis for scalable
systemic-risk mitigation.

\section*{Acknowledgement}
The authors acknowledge support from NSF award No. 2534856.
\bibliography{bibtex}

\begin{thebibliography}{36}
\providecommand{\natexlab}[1]{#1}
\providecommand{\url}[1]{\texttt{#1}}
\expandafter\ifx\csname urlstyle\endcsname\relax
  \providecommand{\doi}[1]{doi: #1}\else
  \providecommand{\doi}{doi: \begingroup \urlstyle{rm}\Url}\fi

\bibitem[Acemoglu et~al.(2012)Acemoglu, Carvalho, Ozdaglar, and
  Tahbaz-Salehi]{Acemoglu2012}
Daron Acemoglu, Vasco~M. Carvalho, Asuman Ozdaglar, and Alireza Tahbaz-Salehi.
\newblock The network origins of aggregate fluctuations.
\newblock \emph{Econometrica}, 80\penalty0 (5):\penalty0 1977--2016, 2012.
\newblock \doi{10.3982/ECTA9623}.

\bibitem[Acemoglu et~al.(2016)Acemoglu, Ozdaglar, and
  Tahbaz-Salehi]{AcemogluOzdaglarTahbazSalehi2016}
Daron Acemoglu, Asuman Ozdaglar, and Alireza Tahbaz-Salehi.
\newblock Networks, shocks, and systemic risk.
\newblock In Yann Bramoull{\'e}, Andrea Galeotti, and Brian~W. Rogers, editors,
  \emph{The Oxford Handbook of the Economics of Networks}, chapter~21, pages
  569--608. Oxford University Press, Oxford, 2016.
\newblock ISBN 978-0-19-994827-7.
\newblock \doi{10.1093/oxfordhb/9780199948277.013.17}.
\newblock URL \url{https://doi.org/10.1093/oxfordhb/9780199948277.013.17}.

\bibitem[Aldous and Steele(2004)]{aldous2004objective}
David Aldous and J.~Michael Steele.
\newblock The objective method: Probabilistic combinatorial optimization and
  local weak convergence.
\newblock In Geoffrey Grimmett, editor, \emph{Probability on Discrete
  Structures}, volume 110 of \emph{Encyclopaedia of Mathematical Sciences},
  pages 1--72. Springer, Berlin, 2004.
\newblock \doi{10.1007/978-3-662-09444-0_1}.

\bibitem[Aldous and Bandyopadhyay(2005)]{aldous2005survey}
David~J. Aldous and Antar Bandyopadhyay.
\newblock A survey of max-type recursive distributional equations.
\newblock \emph{The Annals of Applied Probability}, 15\penalty0 (2):\penalty0
  1047--1110, 2005.
\newblock \doi{10.1214/105051605000000142}.

\bibitem[Amann(1976)]{Amann1976}
Herbert Amann.
\newblock Fixed point equations and nonlinear eigenvalue problems in ordered
  {B}anach spaces.
\newblock \emph{SIAM Review}, 18\penalty0 (4):\penalty0 620--709, 1976.

\bibitem[Amini and Salavati(2024)]{amini2024dynamics}
Hamed Amini and Erfan Salavati.
\newblock Dynamics of cascading losses in locally tree-like networks.
\newblock \emph{SSRN Electronic Journal}, 2024.
\newblock \doi{10.2139/ssrn.4733393}.

\bibitem[Amini et~al.(2016)Amini, Filipovi{\'c}, and
  Minca]{amini2016uniqueness}
Hamed Amini, Damir Filipovi{\'c}, and Andreea Minca.
\newblock Uniqueness of equilibrium in a payment system with liquidation costs.
\newblock \emph{Operations Research Letters}, 44\penalty0 (1):\penalty0 1--5,
  2016.

\bibitem[Amini et~al.(2025)Amini, Cao, and Sulem]{amini2025fire}
Hamed Amini, Zhongyuan Cao, and Agnes Sulem.
\newblock Fire sales, default cascades and complex financial networks.
\newblock \emph{Mathematics and Financial Economics}, 19\penalty0 (2):\penalty0
  225--260, 2025.

\bibitem[Artzner et~al.(1999)Artzner, Delbaen, Eber, and Heath]{Artzner1999}
Philippe Artzner, Freddy Delbaen, Jean-Marc Eber, and David Heath.
\newblock Coherent measures of risk.
\newblock \emph{Mathematical Finance}, 9\penalty0 (3):\penalty0 203--228, 1999.
\newblock \doi{10.1111/1467-9965.00068}.

\bibitem[Ballester et~al.(2006)Ballester, Calv{\'o}-Armengol, and
  Zenou]{BallesterCalvoArmengolZenou2006}
Coralio Ballester, Antoni Calv{\'o}-Armengol, and Yves Zenou.
\newblock Who's who in networks. {W}anted: The key player.
\newblock \emph{Econometrica}, 74\penalty0 (5):\penalty0 1403--1417, 2006.
\newblock \doi{10.1111/j.1468-0262.2006.00709.x}.
\newblock URL \url{https://doi.org/10.1111/j.1468-0262.2006.00709.x}.

\bibitem[Benjamini and Schramm(2001)]{BenjaminiSchramm2001}
Itai Benjamini and Oded Schramm.
\newblock Recurrence of distributional limits of finite planar graphs.
\newblock \emph{Electronic Journal of Probability}, 6:\penalty0 1--13, 2001.
\newblock \doi{10.1214/EJP.v6-96}.

\bibitem[Bollob{\'a}s et~al.(2007)Bollob{\'a}s, Janson, and
  Riordan]{BollobasJansonRiordan2007}
B{\'e}la Bollob{\'a}s, Svante Janson, and Oliver Riordan.
\newblock The phase transition in inhomogeneous random graphs.
\newblock \emph{Random Structures \& Algorithms}, 31\penalty0 (1):\penalty0
  3--122, 2007.
\newblock \doi{10.1002/rsa.20168}.

\bibitem[Cao and Olvera-Cravioto(2020)]{CaoOlveraCravioto2020}
Junyu Cao and Mariana Olvera-Cravioto.
\newblock Connectivity of a general class of inhomogeneous random digraphs.
\newblock \emph{Random Structures \& Algorithms}, 56\penalty0 (3):\penalty0
  722--774, 2020.
\newblock \doi{10.1002/rsa.20892}.

\bibitem[Chen et~al.(2013)Chen, Iyengar, and Moallemi]{chen2013axiomatic}
Chen Chen, Garud Iyengar, and Ciamac~C. Moallemi.
\newblock An axiomatic approach to systemic risk.
\newblock \emph{Management Science}, 59\penalty0 (6):\penalty0 1373--1388, June
  2013.

\bibitem[Eisenberg and Noe(2001)]{Eisenberg2001Systemic}
Larry Eisenberg and Thomas~H. Noe.
\newblock Systemic risk in financial systems.
\newblock \emph{Management Science}, 47\penalty0 (2):\penalty0 236--249, 2001.
\newblock \doi{10.1287/mnsc.47.2.236.9835}.

\bibitem[Elliott et~al.(2014)Elliott, Golub, and
  Jackson]{ElliottGolubJackson2014}
Matthew Elliott, Benjamin Golub, and Matthew~O. Jackson.
\newblock Financial networks and contagion.
\newblock \emph{American Economic Review}, 104\penalty0 (10):\penalty0
  3115--3153, 2014.
\newblock \doi{10.1257/aer.104.10.3115}.

\bibitem[Fraiman et~al.(2023)Fraiman, Lin, and
  Olvera-Cravioto]{FraimanLinOlveraCravioto2023}
Nicolas Fraiman, Tzu-Chi Lin, and Mariana Olvera-Cravioto.
\newblock Stochastic recursions on directed random graphs.
\newblock \emph{Stochastic Processes and their Applications}, 166:\penalty0
  104055, 2023.

\bibitem[Frittelli and Rosazza~Gianin(2005)]{Frittelli2005}
Marco Frittelli and Emanuela Rosazza~Gianin.
\newblock Law invariant convex risk measures.
\newblock In Shigeo Kusuoka and Akira Yamazaki, editors, \emph{Advances in
  Mathematical Economics}, volume~7, pages 33--46. Springer Tokyo, Tokyo, 2005.
\newblock ISBN 978-4-431-27233-5.
\newblock \doi{10.1007/4-431-27233-X_2}.
\newblock URL \url{https://doi.org/10.1007/4-431-27233-X_2}.

\bibitem[Galeotti et~al.(2020)Galeotti, Golub, and
  Goyal]{galeotti2020targeting}
Andrea Galeotti, Benjamin Golub, and Sanjeev Goyal.
\newblock Targeting interventions in networks.
\newblock \emph{Econometrica}, 88\penalty0 (6):\penalty0 2445--2471, 2020.
\newblock \doi{10.3982/ECTA16173}.

\bibitem[Garavaglia et~al.(2020)Garavaglia, van~der Hofstad, and
  Litvak]{garavaglia2020pagerank}
Alessandro Garavaglia, Remco van~der Hofstad, and Nelly Litvak.
\newblock Local weak convergence for {P}age{R}ank.
\newblock \emph{The Annals of Applied Probability}, 30\penalty0 (1):\penalty0
  40--79, 2020.
\newblock \doi{10.1214/19-AAP1494}.

\bibitem[Kallenberg(2017)]{kallenberg2017random}
Olav Kallenberg.
\newblock \emph{Random Measures, Theory and Applications}, volume~77 of
  \emph{Probability Theory and Stochastic Modelling}.
\newblock Springer, Cham, 2017.
\newblock \doi{10.1007/978-3-319-41598-7}.

\bibitem[Kusuoka(2001)]{Kusuoka2001}
Shigeo Kusuoka.
\newblock On law invariant coherent risk measures.
\newblock In Shigeo Kusuoka and Toru Maruyama, editors, \emph{Advances in
  Mathematical Economics}, volume~3, pages 83--95. Springer Japan, Tokyo, 2001.
\newblock ISBN 978-4-431-67891-5.
\newblock \doi{10.1007/978-4-431-67891-5_4}.
\newblock URL \url{https://doi.org/10.1007/978-4-431-67891-5_4}.

\bibitem[Lacker et~al.(2023)Lacker, Ramanan, and Wu]{LackerRamananWu2023}
Daniel Lacker, Kavita Ramanan, and Ruoyu Wu.
\newblock Local weak convergence for sparse networks of interacting processes.
\newblock \emph{The Annals of Applied Probability}, 33\penalty0 (2):\penalty0
  843--888, 2023.
\newblock \doi{10.1214/22-AAP1830}.

\bibitem[Leontief(1966)]{Leontief1966InputOutputEconomics}
Wassily Leontief.
\newblock \emph{Input-Output Economics}.
\newblock Oxford University Press, New York, 1966.

\bibitem[McNeil et~al.(2015)McNeil, Frey, Embrechts,
  et~al.]{mcneil2015quantitative}
Alexander~J McNeil, R{\"u}diger Frey, Paul Embrechts, et~al.
\newblock \emph{Quantitative risk management: Concepts, techniques and tools},
  volume~2.
\newblock Princeton university press Princeton, 2015.

\bibitem[Olvera-Cravioto(2022)]{OlveraCravioto2022}
Mariana Olvera-Cravioto.
\newblock Strong couplings for static locally tree-like random graphs.
\newblock \emph{Journal of Applied Probability}, 59\penalty0 (4):\penalty0
  1261--1285, 2022.
\newblock \doi{10.1017/jpr.2022.17}.

\bibitem[Rockafellar and Uryasev(2002)]{rockafellar2002cvar_general}
R.~Tyrrell Rockafellar and Stanislav Uryasev.
\newblock Conditional value-at-risk for general loss distributions.
\newblock \emph{Journal of Banking \& Finance}, 26\penalty0 (7):\penalty0
  1443--1471, 2002.
\newblock \doi{10.1016/S0378-4266(02)00271-6}.

\bibitem[Shen et~al.(2026)Shen, Van~Oosten, and Wang]{shen2026partial}
Yi~Shen, Zachary Van~Oosten, and Ruodu Wang.
\newblock Partial law invariance and risk measures.
\newblock \emph{Management Science}, 72\penalty0 (7):\penalty0 6251--6266,
  2026.

\bibitem[Tarski(1955)]{Tarski1955}
Alfred Tarski.
\newblock A lattice-theoretical fixpoint theorem and its applications.
\newblock \emph{Pacific Journal of Mathematics}, 5\penalty0 (2):\penalty0
  285--309, 1955.

\bibitem[{U.S. Bureau of Economic Analysis}(2023)]{BEA2023}
{U.S. Bureau of Economic Analysis}.
\newblock Direct domestic requirements, after redefinitions---summary: 2023 (in
  producers' prices).
\newblock Input--Output (After Redefinitions) Interactive Data Application,
  2023.
\newblock URL
  \url{https://apps.bea.gov/iTable/?Categories=AR\&isURI=1\&reqid=1602\&step=2}.
\newblock Accessed March 22, 2026.

\bibitem[{U.S. Bureau of Labor Statistics}(2024)]{BLSQCEW2023}
{U.S. Bureau of Labor Statistics}.
\newblock Employment and wages, annual averages 2023.
\newblock Quarterly Census of Employment and Wages, U.S. Department of Labor,
  2024.
\newblock URL
  \url{https://www.bls.gov/cew/publications/employment-and-wages-annual-averages/2023/}.
\newblock Accessed August 6, 2026.

\bibitem[{U.S. Census Bureau}(2023)]{USCensusCBP2023}
{U.S. Census Bureau}.
\newblock County business patterns: 2023.
\newblock County Business Patterns dataset, 2023.
\newblock URL
  \url{https://www.census.gov/data/datasets/2023/econ/cbp/2023-cbp.html}.
\newblock Accessed August 6, 2026.

\bibitem[van~der Hofstad(2024)]{Hofstad_2024}
Remco van~der Hofstad.
\newblock \emph{Random Graphs and Complex Networks, Volume 2}.
\newblock Cambridge Series in Statistical and Probabilistic Mathematics.
  Cambridge University Press, Cambridge, 2024.
\newblock ISBN 9781107174009.
\newblock \doi{10.1017/9781316795552}.

\bibitem[Villani(2009)]{villani2009optimal}
C{\'e}dric Villani.
\newblock \emph{Optimal transport: old and new}, volume 338.
\newblock Springer Science \& Business Media, Berlin, Heidelberg, 2009.
\newblock ISBN 978-3-540-71049-3.
\newblock \doi{10.1007/978-3-540-71050-9}.

\bibitem[Wu and Amini(2026{\natexlab{a}})]{WuAmini2026LocalApproximation}
Zhecheng Wu and Hamed Amini.
\newblock Local approximation of systemic risk in large sparse economic
  networks.
\newblock In \emph{2026 IEEE Conference on Computational Intelligence in
  Financial Engineering and Economics ({CIFEr})}, Tokyo, Japan,
  2026{\natexlab{a}}.
\newblock URL \url{https://ssrn.com/abstract=6759138}.

\bibitem[Wu and Amini(2026{\natexlab{b}})]{wu2026sampled}
Zhecheng Wu and Hamed Amini.
\newblock Sampled local intervention in large sparse production networks.
\newblock In \emph{2026 IEEE 65th Conference on Decision and Control ({CDC})},
  Honolulu, HI, USA, 2026{\natexlab{b}}. IEEE.
\newblock URL \url{https://ssrn.com/abstract=6570639}.

\end{thebibliography}
\bibliographystyle{plainnat}
\appendix
\section{Technical Details for Local Convergence}
\label{app:technical-details}

We collect the supporting topological and continuity arguments, followed
by type conditioning and the marked inhomogeneous random graph limit.

\subsection{Mark Spaces and the Local Metric}
\label{App:Polish}
\label{app:metric}

We use the depth-\(r\) restriction spaces \(\mathfrak X_r\), metrics
\(d_r\), and local metric \(d_{\mathrm{loc}}\) defined in
Section~\ref{sec:marked_loc_top}. Recall that \(d_r\) compares
depth-\(r\) marked neighborhoods through root-preserving directed
isomorphisms, with distance one when their underlying shapes differ.

\begin{proposition}[Outgoing local space]
\label{prop:explored-local-polish}
The space
$(\widetilde{\mathcal G}_*,d_{\mathrm{loc}})$
is Polish and has the local topology specified in
Section~\ref{sec:marked_loc_top}.
\end{proposition}

\begin{proof}
Let $\mathscr H_r$ be the countable set of admissible unmarked
rooted depth-$r$ shapes. Then
\[
\mathfrak X_r
\cong
\bigsqcup_{H\in\mathscr H_r}
\left(
(\mathsf M^V)^{V(H)}
\times(\mathsf M^E)^{E(H)}
\right)\big/\operatorname{Aut}_o(H).
\]
Each rooted automorphism group $\operatorname{Aut}_o(H)$ is finite
and acts isometrically under the complete maximum product metric.
Hence each component is complete and separable under the quotient
metric $d_r$. Distinct components have distance one, so
$(\mathfrak X_r,d_r)$ is Polish.

For \(0\leq s\leq r\), let
\(\pi_{r,s}:\mathfrak X_r\to\mathfrak X_s\) be the outgoing
depth-\(s\) restriction. These maps satisfy
\(\pi_{r,t}=\pi_{s,t}\circ\pi_{r,s}\) for \(0\leq t\leq s\leq r\), and
\(d_s(\pi_{r,s}a,\pi_{r,s}b)\leq d_r(a,b)\).
Consequently,
\[
\mathfrak X_\infty
=
\bigcap_{0\leq s\leq r<\infty}
\left\{
a\in\prod_{k\geq0}\mathfrak X_k:
\pi_{r,s}(a_r)=a_s
\right\}
\]
is closed in the Polish product equipped with
$D(a,b):=\sum_{r\geq0}2^{-r-1}d_r(a_r,b_r)$.

Given $a\in\mathfrak X_\infty$, choose nested marked
representatives $H_r\in a_r$, relabeled so that
$H_{r+1}|_r=H_r$, where $|_r$ denotes outgoing restriction.
Set $G:=\bigcup_{r\geq0}H_r$. Compatibility gives
$\deg_G^+(v)=\deg_{H_{j+1}}^+(v)<\infty,
\quad j=d_G^+(o,v)$, and $[G,o,\mathcal M(G)]_r=a_r$.
Thus the profile map
$\Phi:\widetilde{\mathcal G}_*\to\mathfrak X_\infty$,
$\Phi(\xi):=([\xi]_r)_{r\geq0}$, is surjective.
By the zero-distance quotient,
$\Phi(\xi)=\Phi(\eta)\iff\xi=\eta$, and
\[
D\bigl(\Phi(\xi),\Phi(\eta)\bigr)
=
\sum_{r\geq0}2^{-r-1}d_r([\xi]_r,[\eta]_r)
=
d_{\mathrm{loc}}(\xi,\eta).
\]
Hence $\Phi$ is an isometric bijection onto a Polish space.
Finally,
$d_{\mathrm{loc}}(\xi_n,\xi)\to0$
if and only if
$d_r([\xi_n]_r,[\xi]_r)\to0$ for every $r\geq0$,
which, by the definition of $d_r$, is precisely the stated
local topology.
\end{proof}

This representation also allows convergence of full rooted laws
to be checked through their finite restrictions.

\begin{corollary}[Convergence through finite restrictions]
\label{coro:projective-convergence}
Fix an integer \(m\geq1\). Set
\(E:=(\widetilde{\mathcal G}_*)^m\) and
\(E_r:=(\mathfrak X_r)^m\), and define \(p_r:E\to E_r\) by
\(p_r(\xi_1,\ldots,\xi_m):=([\xi_1]_r,\ldots,[\xi_m]_r)\).
For Borel probability laws \(Q_n,Q\in\mathscr P(E)\),
\[
Q_n\Rightarrow Q
\quad\Longleftrightarrow\quad
(p_r)_\#Q_n\Rightarrow(p_r)_\#Q
\quad\text{for every }r\geq0,
\]
where \((p_r)_\#Q:=Q\circ p_r^{-1}\) denotes the pushforward.
The same equivalence holds for weak convergence in probability
when \(Q_n\) are random laws and \(Q\) is deterministic.
\end{corollary}

\begin{proof}
The forward implications follow from continuity of $p_r$.
For the converse, set $Q_{n,r}:=(p_r)_\#Q_n$ and
$Q_r:=(p_r)_\#Q$. By
Proposition~\ref{prop:explored-local-polish}, identify $E$
with its closed compatible-profile subspace of
$\prod_{r\geq0}E_r$.
For $\varepsilon>0$, projected tightness gives compact
$K_r\subseteq E_r$ with
$\sup_n Q_{n,r}(K_r^c)\leq\varepsilon2^{-r-1}$.
Then $K:=E\cap\prod_{r\geq0}K_r$ is compact, and
$\sup_n Q_n(E\setminus K)
\leq
\sum_{r\geq0}\sup_n Q_{n,r}(K_r^c)
\leq\varepsilon.$
Thus $(Q_n)$ is relatively weakly compact by Prokhorov's theorem.
Any subsequential limit $\widehat Q$ satisfies
$(p_r)_\#\widehat Q=Q_r=(p_r)_\#Q, $ for all 
$ r\geq0.$
Since $p_s=\pi_{r,s}^{\times m}\circ p_r$ for $s\leq r$,
$\bigcup_{r\geq0}\sigma(p_r)$ is a $\pi$-system generating
$\mathcal B(E)$. Hence $\widehat Q=Q$, and $Q_n\Rightarrow Q$.

For random laws, let $\beta_r$ metrize weak convergence on
$\mathscr P(E_r)$. From any subsequence $(Q_{n_j})$, choose
a further subsequence $(Q_{n_{j_k}})$ such that
\[
\mathbb P\!\left(
\max_{0\leq r\leq k}
\beta_r(Q_{n_{j_k},r},Q_r)>2^{-k}
\right)\leq2^{-k}.
\]
Borel--Cantelli gives
$Q_{n_{j_k},r}\Rightarrow Q_r$ simultaneously for all $r$,
a.s. Applying the deterministic assertion pathwise
yields $Q_{n_{j_k}}\Rightarrow Q$ a.s.
The subsequence criterion therefore gives
$Q_n\overset{\mathbb P}{\Rightarrow}Q$.
\end{proof}

\subsection{Continuity of Finite-Depth Iterates}

We next verify the continuity needed to apply local convergence to
the contractive iterates and the two monotone bounds.

\begin{lemma}[Fixed-depth continuity]
\label{lem:fixed-depth-continuity}
For every \(k\geq0\), the root functional
\(\Phi_k(\xi):=x_o^{k}\), obtained from the local recursion
initialized at zero, is continuous on the contractive
marked-network space.
In the nonexpansive regime, the same holds for
\(\Phi_k^\pm(\xi):=x_o^{\pm,k}\), with initializations
\(x_v^{-,0}=0\) and \(x_v^{+,0}=\mathfrak b(f_v)\).
\end{lemma}

\begin{proof}
For $k=0$, the assertion follows from $\Phi_0=\Phi_0^-=0$ and
$|\mathfrak b(f)-\mathfrak b(g)|\leq\|f-g\|_\infty$.
Fix $k\geq1$ and let $\xi_n\to\xi$. Since
$d_k([\xi_n]_k,[\xi]_k)\leq
2^{k+1}d_{\mathrm{loc}}(\xi_n,\xi)\to0$,
the depth-$k$ shapes eventually agree. Choose root-preserving
shape isomorphisms $\phi_n:[\xi]_k\to[\xi_n]_k$ attaining the
minimum in $d_k$, and pull back all marks and iterates to the common
shape. The response at the corresponding vertex is denoted by
$f_{\phi_n(v)}$.
Writing $m_v(n)=(\tau_v(n),\epsilon_v(n),f_{\phi_n(v)})$, we have
\[
\max_{v\in V_k^+(o)}d_V(m_v(n),m_v)
\vee
\max_{(v,u)\in E_k^+(o)}d_E(C_{vu}(n),C_{vu})
=
d_k([\xi_n]_k,[\xi]_k)\longrightarrow0.
\]
In particular, $\epsilon_v(n)\to\epsilon_v$ and
$C_{vu}(n)\to C_{vu}$.

For $v\in V_{k-1}^+(o)$, the complete outgoing star is revealed.
If it is nonempty, positivity and finiteness give
$S_v(n):=\sum_{u:v\to u}C_{vu}(n)
\to
S_v:=\sum_{u:v\to u}C_{vu}>0,
$ and $
W_{vu}(n)=\frac{C_{vu}(n)}{S_v(n)}
\to W_{vu}.$
Empty stars give zero network input in both graphs.

\smallskip
\noindent\emph{Case 1: Contractive regime.}
Here $f_{\phi_n(v)}\to f_v$ locally uniformly and
$\operatorname{Lip}(f_{\phi_n(v)})\leq L<1$.
For $0\leq\ell\leq k$, set
$\Delta_{\ell,n}
:=
\max_{v\in V_{k-\ell}^+(o)}
|x_v^{\ell}(n)-x_v^{\ell}|,
$ with $
\Delta_{0,n}=0.$
Suppose $\Delta_{\ell,n}\to0$ for some $\ell<k$.
For $v\in V_{k-\ell-1}^+(o)$, all outgoing neighbors belong to
$V_{k-\ell}^+(o)$. With
$z_v^{\ell}:=\sum_{u:v\to u}W_{vu}x_u^{\ell}+\epsilon_v$,
the Lipschitz and row-sum bounds yield
\[
\begin{aligned}
|x_v^{\ell+1}(n)-x_v^{\ell+1}|
\leq
L\Delta_{\ell,n}+L|\epsilon_v(n)-\epsilon_v|+L\sum_{u:v\to u}|W_{vu}(n)-W_{vu}|\,|x_u^{\ell}|
+|f_{\phi_n(v)}(z_v^{\ell})-f_v(z_v^{\ell})|.
\end{aligned}
\]
The last term vanishes by locally uniform convergence at the
fixed input $z_v^{\ell}$. Since $V_{k-\ell-1}^+(o)$ is finite,
$\Delta_{\ell+1,n}\to0$. Induction gives
$|\Phi_k(\xi_n)-\Phi_k(\xi)|
=\Delta_{k,n}\longrightarrow0.$

\smallskip
\noindent\emph{Case 2: Bounded nonexpansive regime.}
Here $\operatorname{Lip}(f_{\phi_n(v)})\leq1$ and
$a_n:=\max_{v\in V_k^+(o)}
\|f_{\phi_n(v)}-f_v\|_\infty\to 0.$
For $\sigma\in\{-,+\}$ and $0\leq\ell\leq k$, set
$
\Delta_{\ell,n}^{\sigma}
:=
\max_{v\in V_{k-\ell}^+(o)}
|x_v^{\sigma,\ell}(n)-x_v^{\sigma,\ell}|.
$
The initializations satisfy
$\Delta_{0,n}^-=0,$ and
$\Delta_{0,n}^+
=
\max_{v\in V_k^+(o)}
|\mathfrak b(f_{\phi_n(v)})-\mathfrak b(f_v)|
\leq a_n\to 0.$
Suppose $\Delta_{\ell,n}^{\sigma}\to0$ for some $\ell<k$.
For $v\in V_{k-\ell-1}^+(o)$, the same estimate with
Lipschitz constant one gives
\[
\begin{aligned}
|x_v^{\sigma,\ell+1}(n)-x_v^{\sigma,\ell+1}|
&\leq
\Delta_{\ell,n}^{\sigma}
+|\epsilon_v(n)-\epsilon_v|+
\sum_{u:v\to u}|W_{vu}(n)-W_{vu}|\,
|x_u^{\sigma,\ell}|+a_n
\longrightarrow0.
\end{aligned}
\]
Finiteness of $V_{k-\ell-1}^+(o)$ yields
$\Delta_{\ell+1,n}^{\sigma}\to0$. Hence, by induction,
$|\Phi_k^\sigma(\xi_n)-\Phi_k^\sigma(\xi)|
=\Delta_{k,n}^{\sigma}\to 0,$ which proves continuity.
\end{proof}

\subsection{Type-Conditioned Rooted Laws}
\label{app:type-conditioned-laws}

Conditioning on a type set preserves local convergence when the
set has positive limiting mass and a \(\mu\)-null boundary.
For a Borel set \(A\subseteq\mathcal T\) with \(\mu(A)>0\),
define
\begin{equation}
\label{eq:type-set-conditional-limit-law}
\mathcal P^A
:=
\frac{\mathcal P(\,\cdot\cap\vartheta^{-1}(A))}{\mu(A)},
\qquad
\mathcal P_n^A
:=
\frac{\mathcal P_n(\,\cdot\cap\vartheta^{-1}(A))}{\mu_n(A)}.
\end{equation}
The second expression applies when \(\mu_n(A)>0\);
set \(\mathcal P_n^A:=\mathcal P^A\) otherwise.

\begin{proposition}[Type-conditioned convergence]
\label{prop:type-set-mixture}
If
\(\mathcal P_n\overset{\mathbb P}{\Rightarrow}\mathcal P\),
\(\mu(A)>0\), and \(\mu(\partial A)=0\), then
\[
\mu_n(A)\xrightarrow{\mathbb P}\mu(A),
\qquad
\mathcal P_n^A
\overset{\mathbb P}{\Rightarrow}\mathcal P^A.
\]
\end{proposition}

\begin{proof}
Continuity of \(\vartheta\) gives
\(\partial(\vartheta^{-1}(A))
\subseteq\vartheta^{-1}(\partial A)\), a
\(\mathcal P\)-null set. Thus, for every
\(h\in C_b(\widetilde{\mathcal G}_*)\),
\(h\mathbf1_{\vartheta^{-1}(A)}\) is bounded and
\(\mathcal P\)-a.e. continuous.
The portmanteau theorem, applied along a.s. convergent
subsequences, yields
\[
\int h\,\mathbf1_{\vartheta^{-1}(A)}\,d\mathcal P_n
\xrightarrow{\mathbb P}
\int h\,\mathbf1_{\vartheta^{-1}(A)}\,d\mathcal P.
\]
Taking \(h\equiv1\) gives
\(\mu_n(A)\xrightarrow{\mathbb P}\mu(A)>0\), so
\(\mathbb P(\mu_n(A)=0)\to0\).
Dividing by the convergent denominators gives
\(\langle\mathcal P_n^A,h\rangle
\xrightarrow{\mathbb P}\langle\mathcal P^A,h\rangle\),
proving the result.
\end{proof}

For finite discrete types, taking \(A=\{t\}\) gives the
type-specific laws \(\mathcal P_n^{(t)}\) and
\(\mathcal P^{(t)}\) whenever \(\mu(\{t\})>0\).

\subsection{Local Convergence of Directed Marked IRGs}
\label{sec:lwc-irg-proof}

The proof of Theorem~\ref{thm:directed-marked-irg-lwc} proceeds from a single exploration to two asymptotically
independent explorations. The latter step upgrades convergence of a
sampled neighborhood to convergence of the empirical rooted law.

\begin{proof}[Proof of Theorem~\ref{thm:directed-marked-irg-lwc}]
Fix an integer \(r\geq0\), and write
\(q_r:\widetilde{\mathcal G}_*\to\mathfrak X_r\),
\(q_r(\xi):=[\xi]_r\), for the restriction map.
We proceed in four steps, first establishing typed local convergence
and then attaching the remaining marks.

\noindent{\bf Step 1: Exploration and the typed offspring law.}
Let \(\mathfrak X_r^{\mathrm{ty}}\) be the depth-\(r\) outgoing
restriction space retaining only the rooted directed graph and its
vertex types. Let \(\mathsf T^{\mathrm{ty}}\) be the typed projection
of the branching process in
Definition~\ref{def:directed-marked-kernel-bp}.
A type-\(s\) parent has offspring-type point measure
\(\Pi_s\sim\operatorname{PRM}(\Lambda_s)\), where
\(\Lambda_s(dt):=\kappa(s,t)\mu(dt)\).
The finite-shape argument in
Proposition~\ref{prop:explored-local-polish}, with vertex-mark space
\(\mathcal T\) and a singleton edge-mark space, shows that
\(\mathfrak X_r^{\mathrm{ty}}\) is Polish.


Choose a root label \(I_n\) uniformly from \([n]\), independently of
the graph. Explore breadth first, exposing complete outgoing rows.
Initially, only the root is discovered. At stage \(h=1,\ldots,r\),
process the vertices first discovered at depth \(h-1\), in a fixed
order. When processing vertex \(v\), reveal all indicators
\(A_{vj}(n)\), \(j\ne v\), record every present arc, and declare
its previously undiscovered endpoints to be depth-\(h\) vertices.
Update the discovered-label set after each row. Each row is processed
once, and no outgoing row of a depth-\(r\) vertex is exposed.
The terminal revealed graph therefore has precisely the vertex and
edge sets
\[
V_r^+(I_n)=\{v:d_{G_n}^+(I_n,v)\leq r\},
\qquad
E_r^+(I_n)=\{(v,u)\in E_n:d_{G_n}^+(I_n,v)<r\},
\]
as required by Definition~\ref{def:U_neighborhood}.
For \(r=0\), only the root and its type are retained.

A label is called used once it has been discovered. A collision
occurs when a present arc points to a label already used before
its row is processed, other than the tail itself. This includes
back edges, cross edges, and two parents discovering the same child:
the latter is detected when the second parent's row is processed.
Since \(A_{vv}(n)=0\), self-loops need not be considered.
Without collisions, the restriction is the rooted discovery tree.

For an integer \(K\geq1\), stop at a collision or as soon as more
than \(K\) labels would be discovered, sending either event to an
isolated absorbing state \(\dagger\).
At the end of stage \(h\), a nonabsorbed state is a rooted typed tree
of depth at most \(h\), with at most \(K\) vertices; its depth-\(h\)
vertices form the next frontier. These states, with \(\dagger\)
adjoined, form a Polish space by the same finite-shape argument.
The finite labels are retained in the conditioning history, but not
in the unlabelled state. Apply the same size cutoff to the branching
process, where collisions do not occur.

Let \(\mathcal H_n\) be the exploration history immediately before
an unprocessed row is exposed, including the type array and root
label. On each compatible history, the parent label \(i\) and used
set \(U\ni i\) are fixed. All off-diagonal indicators in row \(i\)
are still unexamined and, conditionally on \(\mathcal H_n\), are
independent Bernoulli variables with their original parameters.
Define the new-child type measure
\[
\Xi_{n,i}^{U}:=\sum_{j\notin U}A_{ij}(n)\delta_{\tau_j(n)}.
\]
Consider any sequence of such histories with parent labels \(i_n\),
\(s_n:=\tau_{i_n}(n)\to s\), and \(|U_n|\leq K\).
We show that the corresponding conditional laws of
\(\Xi_{n,i_n}^{U_n}\) converge to
\(\operatorname{Law}(\Pi_s)\).

For nonnegative \(g\in C_b(\mathcal T)\), set
\(a_{n,j}:=n^{-1}\kappa(s_n,\tau_j(n))
(1-e^{-g(\tau_j(n))})\).
For sufficiently large \(n\), the truncation in \(p_{ij}(n)\)
is inactive and \(0\leq a_{n,j}\leq\bar\kappa/n\leq1/2\).
Conditional independence gives
\[
\begin{aligned}
\log\mathbb E\bigl[
 e^{-\langle\Xi_{n,i_n}^{U_n},g\rangle}
 \mid\mathcal H_n\bigr]
&=\sum_{j\notin U_n}\log(1-a_{n,j})\\
&=-\sum_{j=1}^n a_{n,j}
 +\sum_{j\in U_n}a_{n,j}
 +\sum_{j\notin U_n}\bigl(\log(1-a_{n,j})+a_{n,j}\bigr)\\
&=-\int_{\mathcal T}\kappa(s_n,t)(1-e^{-g(t)})\mu_n(dt)
 +\mathcal R_n(U_n),
\end{aligned}
\]
where
\[
\mathcal R_n(U_n)
:=\sum_{j\in U_n}a_{n,j}
 +\sum_{j\notin U_n}\bigl(\log(1-a_{n,j})+a_{n,j}\bigr).
\]
The omitted-label term is at most \(K\bar\kappa/n\).
Using \(|\log(1-x)+x|\leq2x^2\) for \(0\leq x\leq1/2\),
the remaining term has absolute value at most
\[
2\sum_{j\notin U_n}a_{n,j}^2
\leq2n(\bar\kappa/n)^2
=\frac{2\bar\kappa^2}{n}.
\]
Consequently,
\[
\sup_{\substack{U_n\ni i_n\\|U_n|\leq K}}
|\mathcal R_n(U_n)|
\leq\frac{K\bar\kappa+2\bar\kappa^2}{n}\longrightarrow0.
\]
This estimate is uniform over all compatible histories and parent
labels satisfying the stated conditions.

Joint continuity of \(\kappa\) implies that
\(\kappa(s_n,\cdot)(1-e^{-g})\) converges uniformly on compact
subsets of \(\mathcal T\) to \(\kappa(s,\cdot)(1-e^{-g})\).
These functions are uniformly bounded. Since
\(\mu_n\Rightarrow\mu\), the measures \(\mu_n\) are uniformly
tight; splitting the integrals over a compact set and its complement
therefore gives
\[
\int\kappa(s_n,t)(1-e^{-g(t)})\mu_n(dt)
\longrightarrow
\int\kappa(s,t)(1-e^{-g(t)})\mu(dt).
\]
The limiting Laplace functional is that of \(\Pi_s\).

For completeness, this convergence holds in the weak topology on
the space \(\mathcal N_f(\mathcal T)\) of finite counting measures,
not merely in a topology based on compactly supported tests.
Indeed, for every compact \(C\subseteq\mathcal T\),
\[
\mathbb E[\Xi_{n,i_n}^{U_n}(\mathcal T)\mid\mathcal H_n]
\leq\bar\kappa,
\qquad
\mathbb E[\Xi_{n,i_n}^{U_n}(C^c)\mid\mathcal H_n]
\leq\bar\kappa\mu_n(C^c).
\]
Uniform tightness of \(\mu_n\) and Markov's inequality allow us
to confine, with arbitrarily high probability, all atoms to a
compact set and their total number to a fixed finite bound.
The corresponding set of counting measures is compact: it is a
finite union of continuous images of finite products of that compact
set. Thus the conditional laws are tight in
\(\mathcal N_f(\mathcal T)\). Laplace functionals on
\(C_b^+(\mathcal T)\) determine finite random-measure laws
\citep{kallenberg2017random}, so every subsequential limit has the
required Poisson law. This proves the asserted conditional convergence.
The same Laplace-functional and tightness argument shows that
\(
s\longmapsto\operatorname{Law}(\Pi_s)
\)
is a Feller kernel.

\noindent{\bf Step 2: Collision control and typed empirical convergence.}
Let \(\mathcal C_{n,r}^{(1)}\) be the event of a collision in the
unrestricted depth-\(r\) exploration from \(I_n\), and set
\(S_{n,r}^{(1)}:=|V_r^+(I_n)|\).
The numbers of rows processed and of labels used before any row
exposure are both bounded by \(S_{n,r}^{(1)}\).

Choose an integer \(n_0\geq1\) such that
\(\bar\kappa/n\leq1/2\) for \(n\geq n_0\).
Given the exploration history, the number of new labels in a row
is a sum of independent Bernoulli variables with parameters
\(p_j\leq\bar\kappa/n\).
Realize each as \(\mathbf1_{\{N_j\geq1\}}\), where the \(N_j\)
are independent Poisson variables with means \(-\log(1-p_j)\).
Then
\[
\sum_j\mathbf1_{\{N_j\geq1\}}\leq\sum_jN_j,
\qquad
\sum_j-\log(1-p_j)\leq2\sum_jp_j\leq2\bar\kappa.
\]
A sequential coupling with fresh dominating offspring variables
therefore bounds the discovered population by that of a
Galton--Watson tree with \(\operatorname{Poisson}(2\bar\kappa)\)
offspring. Write \(\overline Z_0=1\) for its initial generation,
\(\overline Z_\ell\) for subsequent generation sizes, and
\(\overline S_r:=\sum_{\ell=0}^r\overline Z_\ell\).
For \(n\geq n_0\), the coupling gives
\(S_{n,r}^{(1)}\leq\overline S_r\) a.s. Moreover,
\[
\mathbb E[\overline Z_{\ell+1}^2]
=2\bar\kappa\,\mathbb E[\overline Z_\ell]
 +(2\bar\kappa)^2\mathbb E[\overline Z_\ell^2],
\]
so \(\mathbb E[\overline S_r^2]<\infty\) for every fixed \(r\).
Enlarging the constant to cover the finitely many \(n<n_0\),
choose \(C_r<\infty\) such that
\[
\sup_{n\geq1}\mathbb E[(S_{n,r}^{(1)})^2]\leq C_r.
\]
The population of \(\mathsf T^{\mathrm{ty}}\) through depth \(r\)
has the same domination, since \(\Lambda_s(\mathcal T)\leq\bar\kappa\),
and its second moment can also be bounded by \(C_r\).

Suppose the \(t\)th processed row belongs to vertex \(v_t\),
and let \(\mathcal D_t\) and \(\mathcal H_t\) be the used-label
set and history immediately before that row is exposed.
The decision to process this row and its parent label are
\(\mathcal H_t\)-measurable. Conditional on that history,
\[
\mathbb P\bigl(
 A_{v_tj}(n)=1\text{ for some }j\in\mathcal D_t\setminus\{v_t\}
 \mid\mathcal H_t\bigr)
\leq\sum_{j\in\mathcal D_t\setminus\{v_t\}}p_{v_tj}(n)
\leq\frac{\bar\kappa|\mathcal D_t|}{n}.
\]
Let \(R_{n,r}\) be the number of processed rows.
Successive conditioning and the union bound give
\[
\mathbb P(\mathcal C_{n,r}^{(1)})
\leq\frac{\bar\kappa}{n}
 \mathbb E\Bigl[\sum_{t=1}^{R_{n,r}}|\mathcal D_t|\Bigr]
\leq\frac{\bar\kappa}{n}
 \mathbb E[(S_{n,r}^{(1)})^2]
\leq\frac{\bar\kappa C_r}{n}.
\]
The conditional estimate applies before each row is examined;
we never condition on the eventual absence of collisions.
On \((\mathcal C_{n,r}^{(1)})^c\), the discovery tree equals the
entire outgoing depth-\(r\) restriction, not merely a spanning tree,
because every outgoing row at depth strictly below \(r\) has
been fully exposed.

Fix \(K\geq1\), and use the stopped explorations from Step~1.
Suppose nonabsorbed current trees at stage \(h<r\) converge in
\(\mathfrak X_h^{\mathrm{ty}}\).
Their finite shapes eventually agree, and the frontier vertices
can be matched as \(v_{n,1},\ldots,v_{n,q}\) and
\(v_1,\ldots,v_q\), with \(q\leq K\).
If \(i_{n,a}\) is the label of \(v_{n,a}\) and \(U_n\) is the
used-label set at the beginning of the stage, then
\(\tau_{i_{n,a}}(n)\to\tau_{v_a}\) and \(|U_n|\leq K\).
Given any compatible history, the frontier rows have distinct
tails and their indicators remain independent. Step~1 therefore
gives, under these conditional laws,
\[
\bigl(\Xi_{n,i_{n,a}}^{U_n}\bigr)_{a=1}^q
\Rightarrow(\Pi_a)_{a=1}^q,
\qquad
\Pi_a\sim\operatorname{PRM}(\Lambda_{\tau_{v_a}}),
\]
with independent limiting point measures. For \(q=0\), the
current tree is retained.

The offspring measures in this display use the common set \(U_n\)
from the beginning of the stage. They may therefore share a new
finite-graph label, which is a collision. Conditional on the
current history \(\mathcal H\),
\[
\begin{aligned}
\mathbb P(\text{collision during the next stage}\mid\mathcal H)
&\leq\sum_{a=1}^q\sum_{j\in U_n\setminus\{i_{n,a}\}}
       p_{i_{n,a}j}(n)
+\sum_{1\leq a<b\leq q}\sum_{j\notin U_n}
       p_{i_{n,a}j}(n)p_{i_{n,b}j}(n)\\
&\leq\frac{qK\bar\kappa}{n}
       +\binom q2\frac{\bar\kappa^2}{n}.
\end{aligned}
\]
The two terms account for arcs to used labels and common new children,
respectively. Apart from this event, the next generation is obtained
by attaching a distinct child for each atom of the corresponding
parent's offspring measure.

This attachment is continuous. For deterministic finite counting
measures \(\nu_{n,a}\to\nu_a\) weakly, their total masses converge
and hence, being integers, eventually agree. Their atoms can then
be matched, with multiplicities respected, so that all corresponding
types converge. Attaching these children to matched parents preserves
convergence in the typed local topology. The resulting unlabelled
tree does not depend on the temporary enumerations.
The total number of vertices is eventually constant along a
convergent sequence, so sending trees with more than \(K\) vertices
to \(\dagger\) is also continuous.

We can now induct over \(h\). At \(h=0\), convergence follows from
\(\tau_{I_n}(n)\sim\mu_n\Rightarrow\mu\).
For a bounded continuous test function of the next-stage stopped
tree, its expectation under the limiting branching extension is a
bounded continuous function of the current tree, by the Feller
offspring kernel and attachment continuity.
The finite-graph conditional expectation converges to this function
uniformly on compact sets of current trees and over compatible
histories. Otherwise, a sequence of violating histories with current
trees in a compact set would have a convergent subsequence,
contradicting the one-stage convergence just proved.
The induction hypothesis gives tightness of the current-state laws.
On a compact set use the uniform conditional-expectation bound,
and on its complement use boundedness. The resulting error tends
to zero in \(L^1\). Taking expectations and applying current-state
weak convergence to the limiting continuous function proves
next-stage convergence. At \(\dagger\), both processes remain
absorbed. This establishes stopped convergence through depth \(r\).

The terminal stopped tree agrees with
\([G_n^\kappa,I_n,\tau(n)]_r\) outside
\(\mathcal C_{n,r}^{(1)}\cup\{S_{n,r}^{(1)}>K\}\), whose probability
is at most
\[
\frac{\bar\kappa C_r}{n}+\frac{C_r}{K^2}.
\]
The limiting size-cutoff probability is at most \(C_r/K^2\).
Extend a bounded continuous test function on
\(\mathfrak X_r^{\mathrm{ty}}\) by zero at \(\dagger\).
Its expectations for the original and stopped explorations differ
by at most its supremum norm times the respective stopping
probability. Letting first \(n\to\infty\) and then \(K\to\infty\)
therefore gives
\[
[G_n^\kappa,I_n,\tau(n)]_r
\Rightarrow[\mathsf T^{\mathrm{ty}}]_r.
\]

\subparagraph{Two-root convergence.}
Choose \(J_n\) independently and uniformly from \([n]\), independently
of \(I_n\) and the graph. Explore from both roots, declaring both
initial labels used and using one common used-label set.
The argument follows the two-exploration decoupling principle for
directed random graphs; see \citet{CaoOlveraCravioto2020}.
Declare a joint collision if the roots coincide or a present arc
reaches a used label, including a label discovered by the other
exploration. Stop at the first such event, denoted by
\(\mathcal C_{n,r}^{(2)}\), or after both explorations finish.
This event includes any internal collision and every intersection
of the two depth-\(r\) neighborhoods.

Before a joint collision, the numbers of exposed rows and used
labels are bounded by
\(S_{n,r}^{\mathrm{tot}}:=|V_r^+(I_n)|+|V_r^+(J_n)|\).
The one-root moment bound gives
\(\mathbb E[(S_{n,r}^{\mathrm{tot}})^2]\leq4C_r\), without
assuming independence of the two neighborhoods.
Since \(\mathbb P(I_n=J_n)=1/n\), the same conditional exposure
estimate yields
\[
\mathbb P(\mathcal C_{n,r}^{(2)})
\leq\frac1n+\frac{\bar\kappa}{n}
             \mathbb E[(S_{n,r}^{\mathrm{tot}})^2]
\leq\frac{c_{2,r}}{n},
\qquad c_{2,r}:=1+4\bar\kappa C_r.
\]
In particular,
\[
\mathbb P\bigl(V_r^+(I_n)\cap V_r^+(J_n)\ne\varnothing\bigr)
\leq\frac{c_{2,r}}{n}\longrightarrow0.
\]

For joint convergence, also stop if either tree would exceed \(K\)
vertices, representing any stopping event by one isolated state.
Apply the same size rule to two independent copies of
\(\mathsf T^{\mathrm{ty}}\).
The initial types have joint law \(\mu_n^{\otimes2}\Rightarrow
\mu^{\otimes2}\). Sending coincident roots to the absorbing state
changes this initial law on an event of probability \(1/n\).
Before stopping, there are at most \(2K\) used labels; given the
history, all unprocessed frontier rows have distinct tails and
independent indicators. The one-generation convergence, stage
collision estimate, and attachment argument therefore apply jointly,
with \(2K\) in place of \(K\).
The same induction proves convergence of the stopped pair to the
correspondingly stopped pair of independent branching processes.
Again, independence is used conditional on the current history,
not conditional on a terminal no-collision event.

The stopped pair differs from the original pair only on an event
of probability at most \(c_{2,r}/n+2C_r/K^2\).
The limiting stopping probability is at most \(2C_r/K^2\).
Removing the cutoff as above gives
\[
\bigl([G_n^\kappa,I_n,\tau(n)]_r,
      [G_n^\kappa,J_n,\tau(n)]_r\bigr)
\Rightarrow
\bigl([\mathsf T_1^{\mathrm{ty}}]_r,
      [\mathsf T_2^{\mathrm{ty}}]_r\bigr),
\]
where the two limiting typed branching processes are independent.

Define
\[
\mathcal P_{n,r}^{\mathrm{ty}}
:=\frac1n\sum_{i=1}^n\delta_{[G_n^\kappa,i,\tau(n)]_r},
\qquad
\mathcal P_r^{\mathrm{ty}}
:=\operatorname{Law}([\mathsf T^{\mathrm{ty}}]_r).
\]
For real-valued \(F\in C_b(\mathfrak X_r^{\mathrm{ty}})\),
conditional uniform sampling gives
\[
\begin{aligned}
\mathbb E\langle\mathcal P_{n,r}^{\mathrm{ty}},F\rangle
&=\mathbb E F([G_n^\kappa,I_n,\tau(n)]_r),\\
\mathbb E\bigl[\langle\mathcal P_{n,r}^{\mathrm{ty}},F\rangle^2\bigr]
&=\mathbb E\bigl[
F([G_n^\kappa,I_n,\tau(n)]_r)
F([G_n^\kappa,J_n,\tau(n)]_r)\bigr].
\end{aligned}
\]
The one- and two-root convergences imply convergence of these
moments to \(a\) and \(a^2\), respectively, where
\(a:=\langle\mathcal P_r^{\mathrm{ty}},F\rangle\).
Expanding the squared error therefore yields
\[
\langle\mathcal P_{n,r}^{\mathrm{ty}},F\rangle
\xrightarrow{L^2}\langle\mathcal P_r^{\mathrm{ty}},F\rangle.
\]
Since \(\mathfrak X_r^{\mathrm{ty}}\) is Polish, choose a countable
convergence-determining class of bounded continuous functions.
From any subsequence, a diagonal extraction gives a further
subsequence with simultaneous almost-sure convergence of all
integrals in this class. The projected laws converge weakly on
that probability-one event. The subsequence criterion gives
\[
\mathcal P_{n,r}^{\mathrm{ty}}
\overset{\mathbb P}{\Rightarrow}\mathcal P_r^{\mathrm{ty}}.
\]

\noindent {\bf Step 3: Type-dependent independent marking.}
The continuous forgetful map
\(\pi_r^{\mathrm{ty}}:\mathfrak X_r\to\mathfrak X_r^{\mathrm{ty}}\)
removes shocks, responses, and exposures, retaining the graph
and structural vertex types.
For \(h\in\mathfrak X_r^{\mathrm{ty}}\), choose a finite labelled
representative \(H=(V,E,o,\tau_H)\), and independently attach
\[
(\epsilon_v,f_v)\sim\mathsf K^V_{\tau_H(v)},\quad v\in V,
\qquad
C_{vu}\sim\mathsf K^E_{\tau_H(v),\tau_H(u)},\quad (v,u)\in E.
\]
Let \(K_r(h,\cdot)\) be the law of the resulting marked
isomorphism class. A change of representative only permutes the
mark coordinates in a type-preserving way, so this law is
well defined. It is supported on the fiber above \(h\):
\[
K_r\bigl(h,(\pi_r^{\mathrm{ty}})^{-1}(\{h\})\bigr)=1.
\]

The kernel \(K_r\) is Feller. If \(h_n\to h\), the restrictions
have an eventually common finite directed shape after matching,
and all corresponding types converge. By
Assumption~\ref{ass:feller-marking-kernels}, the finite products of
vertex- and edge-mark laws converge weakly.
Attaching the type coordinates and passing to the marked
isomorphism class are continuous, giving
\(K_r(h_n,\cdot)\Rightarrow K_r(h,\cdot)\).
In particular, \(K_r\) is a Borel probability kernel and, for every
\(F\in C_b(\mathfrak X_r)\),
\[
\overline F(h):=\int_{\mathfrak X_r}F(\xi)K_r(h,d\xi)
\]
belongs to \(C_b(\mathfrak X_r^{\mathrm{ty}})\), with
\(\|\overline F\|_\infty\leq\|F\|_\infty\).
For \(Q^{\mathrm{ty}}\in\mathscr P(\mathfrak X_r^{\mathrm{ty}})\),
define the marked lift by
\[
(Q^{\mathrm{ty}}K_r)(A):=\int K_r(h,A)Q^{\mathrm{ty}}(dh).
\]
Then
\[
(\pi_r^{\mathrm{ty}})_\#(Q^{\mathrm{ty}}K_r)=Q^{\mathrm{ty}},
\qquad
(q_r)_\#\mathcal P=\mathcal P_r^{\mathrm{ty}}K_r,
\]
where the second identity follows from the marking rule in
Definition~\ref{def:directed-marked-kernel-bp}.

\noindent{\bf Step 4: Empirical marked convergence and the full local limit.}
For real-valued \(F\in C_b(\mathfrak X_r)\), set
\[
Z_{n,r}(F):=\frac1n\sum_{i=1}^n
F([G_n^\kappa,i,\mathcal M_n]_r)
=\langle(q_r)_\#\mathcal P_n,F\rangle.
\]
Let \(H_n^{\mathrm{ty}}:=(G_n^\kappa,\tau(n))\) be the full typed
graph. Conditional on this graph, the marking rule gives
\[
\begin{aligned}
\mathbb E[Z_{n,r}(F)\mid H_n^{\mathrm{ty}}]
&=\langle\mathcal P_{n,r}^{\mathrm{ty}},\overline F\rangle\\
&\xrightarrow{\mathbb P}
\langle\mathcal P_r^{\mathrm{ty}},\overline F\rangle
=\langle(q_r)_\#\mathcal P,F\rangle.
\end{aligned}
\]
This uses the typed empirical convergence of Step~2 and the Feller
property established in Step~3.

It remains to control the residual mark randomness. Write
\(Y_{n,i}:=F([G_n^\kappa,i,\mathcal M_n]_r)\). Then
\[
\operatorname{Var}(Z_{n,r}(F)\mid H_n^{\mathrm{ty}})
=\frac1{n^2}\sum_{i,j=1}^n
\operatorname{Cov}(Y_{n,i},Y_{n,j}\mid H_n^{\mathrm{ty}}).
\]
Every mark used by \(Y_{n,i}\) belongs to a vertex in \(V_r^+(i)\)
or an edge with both endpoints in \(V_r^+(i)\).
If \(V_r^+(i)\cap V_r^+(j)=\varnothing\), the two observables use
disjoint mark families and are conditionally independent by
Assumption~\ref{ass:feller-marking-kernels}. Their conditional
covariance is therefore zero. For overlapping pairs, boundedness
gives the sufficient estimate
\(\bigl|\operatorname{Cov}(Y_{n,i},Y_{n,j}\mid H_n^{\mathrm{ty}})\bigr|
\leq4\|F\|_\infty^2\). Hence
\[
\operatorname{Var}(Z_{n,r}(F)\mid H_n^{\mathrm{ty}})
\leq\frac{4\|F\|_\infty^2}{n^2}
\sum_{i,j=1}^n
\mathbf1_{\{V_r^+(i)\cap V_r^+(j)\ne\varnothing\}}.
\]
Averaging and using the independent uniform roots from Step~2,
\[
\begin{aligned}
\mathbb E\bigl[
|Z_{n,r}(F)-\mathbb E[Z_{n,r}(F)\mid H_n^{\mathrm{ty}}]|^2
\bigr]
=&\mathbb E\bigl[
\operatorname{Var}(Z_{n,r}(F)\mid H_n^{\mathrm{ty}})\bigr]\\
\leq&4\|F\|_\infty^2
\mathbb P\bigl(V_r^+(I_n)\cap V_r^+(J_n)\ne\varnothing\bigr)
\leq\frac{4\|F\|_\infty^2c_{2,r}}{n}\longrightarrow0.
\end{aligned}
\]
Together with conditional-mean convergence, this proves
\[
\langle(q_r)_\#\mathcal P_n,F\rangle
\xrightarrow{\mathbb P}\langle(q_r)_\#\mathcal P,F\rangle
\qquad\text{for every }F\in C_b(\mathfrak X_r).
\]

Choose a countable convergence-determining class in
\(C_b(\mathfrak X_r)\). From any subsequence, a diagonal
extraction gives a further subsequence along which all integrals
in this class converge a.s. The projected laws then converge
weakly on a probability-one event. The subsequence criterion gives
\[
(q_r)_\#\mathcal P_n
\overset{\mathbb P}{\Rightarrow}(q_r)_\#\mathcal P
\qquad\text{for every fixed }r\geq0.
\]
Corollary~\ref{coro:projective-convergence}, with \(m=1\), now yields
\(\mathcal P_n\overset{\mathbb P}{\Rightarrow}\mathcal P\), as claimed.
\end{proof}

\end{document}